\documentclass{amsart}
\usepackage{amsrefs,amssymb,amsxtra,xr-hyper}
\usepackage{signs}
\usepackage{mathrsfs}
\newcommand\TitsNpp{\me{\mathscr N''}}
\usepackage[all]{xy}
\usepackage{stmaryrd}
\domanymulti\makecite
	({\bourbaki{bourbaki:lie-gp+lie-alg_4-6}}
	{\carter{carter:weyl}}
	{\springer{springer:lag}}
	{\tits{tits:extended}})
\author{Loren Spice}
\address{Texas Christian University \\ Fort Worth, TX 76109}
\email{l.spice@tcu.edu}
\thanks{The author was partially supported by National Science Foundation Postdoctoral Fellowship award DMS-0503107, and by Simons Foundation Collaboration Grant 246066.}

\title{On counting orbits in root systems}
\date\today
\begin{document}
\begin{abstract}
The computation of the characters of supercuspidal representations of a \(p\)-adic group \cite{adler-spice:explicit-chars} involves some \(4\)th roots of unity whose values are defined in terms of orbits of the Galois group of a \(p\)-field on a root system.  The part of the definition that is of interest in the verification of stability of character sums \cite{debacker-spice:stability} involves just the parity of the number of Galois orbits.  In this paper, we re-cast the definition (nearly) in terms only of the abstract action of a pair of automorphisms on a root system, and compute it by a series of reductions in all cases.
\end{abstract}

\maketitle
\setcounter{tocdepth}2
\tableofcontents
\listoftables

\section{Introduction}
\label{sec:intro}

In the computation of characters of supercuspidal representations of \(p\)-adic groups \xcite{adler-spice:explicit-chars}*{Theorem \xref{thm:full-char}}, there arise, in addition to some now reasonably well understood positive constants \xcite{debacker-spice:stability}*{Lemma \xref{lem:char-summand}}, certain somewhat mysterious \(4\)th roots of unity.  Loosely speaking, these roots of unity have a `ramified part', of order dividing \(4\), which is not expected to change across a supercuspidal L-packet, and an `unramified part', of order dividing \(2\), which is expected to give stability when summing the characters in a supercuspidal L-packet.  See \xcite{debacker-spice:stability}*{Definition \xref{defn:st-sign}} for definitions, and \xcite{debacker-spice:stability}*{Theorem \xref{thm:stable}} for a verification of these expectations for the positive-depth, `unramified' supercuspidal L-packets constructed by Reeder \cite{reeder:sc-pos-depth}*{\S6.6}.  Our goal, here and in \cite{spice:signs-alg}, is to understand the unramified part in order to make progress towards a construction of `\emph{ramified}' supercuspidal L-packets.

The definition of the unramified part is stated in a way that seems to involve the power of the Bruhat--Tits machinery \cites{bruhat-tits:reductive-groups-1,bruhat-tits:reductive-groups-2} in an essential way; but, in fact, one may re-cast the definition so that it involves only the theory of a finite group of Lie type (thought of as reductive quotients of \(p\)-adic groups \cite{tits:corvallis}*{\S3.5}) equipped with an algebraic automorphism \(\theta\) (reflecting the action of the inertia subgroup of the Galois group).  A similar point of view is evident in \cite{reeder-levy-yu-gross:gradings}*{\S1, p.~1126} and \cite{reeder-yu:epipelagic}*{\S4.1}; I thank Tasho Kaletha for suggesting that it might be useful here.  Further, it turns out that the definition \emph{almost} does not even involve the ambient group at all, only the problem of counting the orbits for a certain action on its root system.  The caveat `almost' comes from the fact that we are not dealing with the full root system, but rather with a subset \(\Root_\theta\) associated to the automorphism \(\theta\).

Even this caveat does not arise if we work with automorphisms \(\theta\) that act by an \emph{elliptic} automorphism \(w\) of the root system.  In this case, though we must still use a choice of \(\theta\) to define \(\Root_\theta\), the actual choice does not matter, and so we obtain a set \(\Root_w\) depending only on \(w\).  Once this definition has been made, the remainder of the sign computation takes place entirely in the abstract setting of a group action on the root system.

Although the original statement of the problem involves counting orbits on \(\Root_\theta\), it turns out to be more convenient to consider the sign of a certain permutation of a quotient of \(\Root_\theta\).  This is closely related to the orbit-counting problem (see Lemma \ref{lem:perm-sign}), but more amenable to calculation because it is a `multiplicative question', in an appropriate sense.

In this paper, we carry out this abstract computation as follows.  There is a case-by-case computation at the heart of our work (see \S\S\ref{sec:F4} and \ref{sec:En}), but, before getting there, we need to reduce to a manageable number of cases.  We begin by discussing generalities on permutations in \S\ref{sec:perm}, particularly the permutation arising by multiplication by a unit on \(\Z/n\Z\), which is used in the classical-group computations of \S\ref{sec:classical}.  See Proposition \ref{prop:Legendre}.  Next we recall some generalities on root systems and their automorphisms in \S\S\ref{sec:root-system}--\S\ref{sec:parabolic}; particularly, we discuss `Levi descent' for automorphisms of root systems.  See Corollary \ref{cor:Borel}.

In \S\ref{sec:R_theta}, we define the set \(\Root_\theta\) (Definition \ref{defn:R_theta}) and discuss some of its properties; see particularly Definition \ref{defn:Rw} and Lemma \ref{lem:odd-power}.

In \S\ref{sec:root-sign}, we define the sign \(\Legend\sigma\theta\) (Definition \ref{defn:root-sign}), and show that its computation may be reduced to the case where \(\theta\) acts by an elliptic (Proposition \ref{prop:reduce-elliptic}) automorphism of an irreducible (Remark \ref{rem:reduce-irreducible}) root system of \(2\)-power order (Lemma \ref{lem:reduce-2-power}).  As a preliminary example, we compute it in case \(\theta\) acts by \(-1\) on the root system; see Proposition \ref{prop:minus-sign}.  The computation for all the classical root systems can be handled almost uniformly; we do so in \S\ref{sec:classical}.  See in particular \eqref{eq:BCD:Coxeter}--\eqref{eq:BCD:norm}.

Finally we have reduced to a manageable number of cases, namely, the elliptic conjugacy classes of automorphisms of exceptional root systems of \(2\)-power order.  See Remark \ref{rem:elliptic-list}.  For \Gn, the only such automorphism is \(-1\), which is handled (in a uniform way) in Proposition \ref{prop:minus-sign}.  We handle the possibilities for \Fn in \S\ref{sec:F4}, and for \En in \S\ref{sec:En}.  The results are summarised in Tables \ref{tab:An}--\ref{tab:E6}.

It is a pleasure to thank Jeff Adler, Stephen DeBacker, Kyle Petersen, and John Stembridge for many useful conversations.

\section{Preliminaries}

\numberwithin{equation}{subsection}

\subsection{Permutations}
\label{sec:perm}

The notation \(\sgn\), with a single subscript, will occur in two different ways in this paper (see Definition \ref{defn:sgn-Root-A}); we rely on context to distinguish them.

\begin{defn}
\label{defn:sgn-perm}
If \(X\) is any finite set, then we write \mnotn{\Sgp_X} for the symmetric group on \(X\), and \mnotn{\sgn_X} for the unique homomorphism \anonmap{\Sgp_X}{\sgen{-1}} that takes the value \(-1\) at any transposition.  We usually abbreviate \(\Sgp_{\Z/n\Z}\) to \mnotn{\Sgp_n}, and \(\sgn_{\Z/n\Z}\) to \mnotn{\sgn_n}.
\end{defn}

The following lemma is easy, but crucial to the main computational tool of this paper, which is to replace an orbit-counting problem by the problem of computing the sign of a permutation.

\begin{lem}
\label{lem:perm-sign}
For all finite sets \(X\) and all \(w \in \Sgp_X\), we have that
\[
\sgn_X(w) = (-1)^{\card X}(-1)^{\card{\sgen w\bslash X}}.
\]
\end{lem}

\begin{proof}
Both sides equal \(\prod (-1)^{\card\omega - 1}\), where the product runs over the orbits \(\omega\) of \(w\) on \(X\).
\end{proof}

In our computation of signs associated to Galois actions on classical root systems (see \S\ref{sec:classical}), we shall often have occasion to compute the sign of permutations arising by multiplication.

\begin{defn}
\label{defn:sgn+-}
If \(n\) and \(q\) are relatively prime integers and \(\varepsilon \in \sset{\pm1}\), then we denote the multiplication-by-\(q\) map on \(\sgen\varepsilon\bslash\Z/n\Z\) again by \(q\), and put \(\mnotn{\sgn^\varepsilon_n(q)} = \sgn_{\sgen\varepsilon\bslash\Z/n\Z}(q)\).
\end{defn}

The homomorphisms \(\sgn^\pm_n\) turn out to be closely related to the Jacobi symbol
\mnonotn{\displaystyle\Legend q n}%
\Legend\cdot n when \(n\) is odd.  We begin with a reduction result that computes \(\sgn^\pm_n\) in some special cases, then apply it in Proposition \ref{prop:Legendre} to compute \(\sgn^\pm_n\) in all cases.

\begin{prop}
\toplabel{prop:pre-Legendre}
Suppose that \(n\) and \(q\) are relatively prime integers.
\begin{enumerate}
\item\sublabel{p-power} If \(p\) is an odd prime, \(k \in \Z_{\ge 0}\), and \(n = p^k\), then
\[
\sgn^+_n(q) = \Legend q n
\qandq
\sgn^-_n(q) = \begin{cases}
\Legend q n, & p \equiv 1 \pmod4  \\
1,           & p \equiv 3 \pmod4.
\end{cases}
\]
\item\sublabel{2-power} If \(k \in \Z_{\ge 0}\) and \(n = 2^k\), then
\[
\sgn^+_n(q) = \begin{cases}
(-1)^{(q - 1)/2}, & k > 1,           \\
1,                & \text{otherwise}
\end{cases}
\qandq
\sgn^-_n(q) = \begin{cases}
\Legend n q, & k > 1,              \\
1,             & \text{otherwise.}
\end{cases}
\]
\item\sublabel{CRT} If \(m \in \Z\) is relatively prime to \(n\) and \(q\), then
\[
\sgn^+_{m n}(q) = \sgn^+_m(q)^n\sgn^-_n(q)^m
\]
and
\[
\sgn^-_{m n}(q) = \sgn^+_m(q)^{\rup{(n - 1)/2}}\sgn^-_m(q)^n\sgn^+_n(q)^{\rup{(m - 1)/2}}\sgn^-_n(q)^m.
\]
\end{enumerate}
\end{prop}

\begin{proof}
We begin by proving \subpref{p-power} and \subpref{2-power}.  Suppose that \(\varepsilon \in \sset{\pm1}\), \(k \in \Z_{> 0}\), and \(p\) is \emph{any} prime, and put \(n = p^k\) and \(n' = p^{k - 1}\).  Then \((\Z/n\Z)\mult\) and \(p\Z/n\Z\) form a partition of \(\Z/n\Z\) into \(\varepsilon\)- and \(q\)-stable subsets, and \map p{\Z/n'\Z}{p\Z/n\Z} is an \(\varepsilon\)- and \(q\)-equivariant bijection.  Thus,
\[
\sgn^\varepsilon_n(q) = \sgn^\varepsilon_{n'}(q)\sgn_{\sgen\varepsilon\bslash(\Z/n\Z)\mult}(q).
\]
By Lemma \ref{lem:perm-sign},
\[
\sgn_{\sgen\varepsilon\bslash(\Z/n\Z)\mult}(q) = (-1)^{\impindx{\sgen\varepsilon} - \impindx{\sgen{q, \varepsilon}}},
\]
where we have written \impindx G for the index in \((\Z/n\Z)\mult\) of a subgroup \(G\).

If
\begin{itemize}
\item \(p\) is odd and
	\begin{itemize}
	\item \(\varepsilon = 1\) or
	\item \(p \equiv 1 \pmod4\), or
	\end{itemize}
\item \(p = 2\) and
	\begin{itemize}
	\item \(\varepsilon = 1\) and \(k = 2\)
	or
	\item \(\varepsilon = -1\) and \(k > 2\),
	\end{itemize}
\end{itemize}
then the group \(\sgen\varepsilon\bslash(\Z/n\Z)\mult\) has a unique index-\(2\) subgroup, so \anonmapto q{(-1)^{\impindx{\sgen\varepsilon} - \impindx{\sgen{q, \varepsilon}}} = (-1)^{\impindx{\sgen{q, \varepsilon}}}} is its unique non-trivial quadratic character.  That is,
\begin{itemize}
\item if \(p\) is odd, and \(\varepsilon = 1\) or \(p \equiv 1 \pmod4\), then
\[
(-1)^{\impindx{\sgen\varepsilon} - \impindx{\sgen{q, \varepsilon}}} = \Legend q p = \Legend q n\Legend q{n'};
\]
\item if \(p = 2\), \(\varepsilon = 1\), and \(k = 2\), then
\[
(-1)^{\impindx{\sgen\varepsilon} - \impindx{\sgen{q, \varepsilon}}} = (-1)^{(q - 1)/2};
\]
and
\item if \(p = 2\), \(\varepsilon = -1\), and \(k > 2\), then
\[
(-1)^{\impindx{\sgen\varepsilon} - \impindx{\sgen{q, \varepsilon}}} = \Legend 2 q = \Legend n q\Legend{n'}q.
\]
\end{itemize}

If
\begin{itemize}
\item \(p\) is odd, \(\varepsilon = -1\), and \(p \equiv 3 \pmod4\),
or
\item \(p = 2\) and
	\begin{itemize}
	\item \(\varepsilon = 1\) and \(k = 1\) or
	\item \(\varepsilon = -1\) and \(k \le 2\),
	\end{itemize}
\end{itemize}
then the group \(\sgen\varepsilon\bslash(\Z/n\Z)\mult\) has odd order, so that \impindx{\sgen\varepsilon} and \impindx{\sgen{q, \varepsilon}} are both odd, whence \((-1)^{\impindx{\sgen\varepsilon} - \impindx{\sgen{q, \varepsilon}}} = 1\).  In particular, if \(p = 2\), \(\varepsilon = -1\), and \(k = 2\), then
\[
(-1)^{\impindx{\sgen\varepsilon} - \impindx{\sgen{q, \varepsilon}}} = 1 = \Legend4 q.
\]

By the preceding three paragraphs, \subpref{p-power} and \subpref{2-power} now hold by induction on \(k\) (the results being obvious if \(k = 0\)).

Next we prove \subpref{CRT}.  We use the ring isomorphism
\begin{equation}
\tag{$*$}
\sublabel{eq:CRT}
\Z/m n\Z \cong \Z/m\Z \oplus \Z/n\Z.
\end{equation}
Since the desired equalities are symmetric in \(m\) and \(n\), and multiplicative in \(q\), it suffices to prove them for \(q \in (\Z/m\Z)\mult\).  The natural projection \anonmap{\Z/m n\Z}{\Z/m\Z} is a \(q\)-equivariant map with fibres of cardinality \(n\), so
\[
\sgn^+_{m n}(q) = \sgn^+_m(q)^n = \sgn^+_m(q)^n\sgn^+_n(q)^m.
\]

Now fix \(a \in \Z/m n\Z\).  If \(n \mid 2a\), \ie, if \subeqref{eq:CRT} carries \(a\) into \(\Z/m\Z \oplus (\Z/n\Z)[2]\), where \(R[2]\) denotes the ring of \(2\)-torsion elements of \(R\), then \anonmapto{\sgen{-1}\dota q^i a}{\sgen{-1}\dota(q^i a + m\Z)} is a bijection between the \(q\)-orbit through \(\sgen{-1}\dota a \in \sgen{-1}\bslash\Z/m n\Z\) and the \(q\)-orbit through \(\sgen{-1}\dota(a + m\Z) \in \sgen{-1}\bslash\Z/m\Z\).  Otherwise, \anonmap{\sgen{-1}\dota q^i a}{q^i a + m\Z} is a bijection between the \(q\)-orbit through \(\sgen{-1}\dota a \in \sgen{-1}\bslash\Z/m n\Z\) and the \(q\)-orbit through \(a + m\Z \in \Z/m\Z\).  That is, the cycle type of \(q\) on \(\sgen{-1}\bslash\Z/m n\Z\) is the union of \(\card{\sgen{-1}\bslash\Z/n\Z} - \card{(\Z/n\Z)[2]}\) copies of the cycle type of \(q\) on \(\Z/m\Z\), and \card{(\Z/n\Z)[2]} copies of the cycle type of \(q\) on \(\sgen{-1}\bslash\Z/m\Z\); so
\begin{align*}
\sgn^-_{m n}(q)
={} & \sgn^+_m(q)^{\card{\sgen{-1}\bslash\Z/n\Z} - \card{(\Z/n\Z)[2]}}\sgn^-_m(q)^{\card{(\Z/n\Z)[2]}} \\
={} & \sgn^+_m(q)^{\card{\sgen{-1}\bslash\Z/n\Z} - \card{(\Z/n\Z)[2]}}\sgn^-_m(q)^{\card{(\Z/n\Z)[2]}} \times{} \\
    & \qquad\sgn^+_n(q)^{\card{\sgen{-1}\bslash\Z/m\Z} - \card{(\Z/m\Z)[2]}}\sgn^-_n(q)^{\card{(\Z/m\Z)[2]}}.
\end{align*}
\subpref{CRT} follows upon observing that \(\card{\sgen{-1}\bslash\Z/a\Z} = \rup{\frac1 2(a + 1)}\) and \(\card{(\Z/a\Z)[2]} \equiv a \pmod2\) for \(a \in \Z \setminus \sset0\).
\end{proof}

\begin{prop}
\label{prop:Legendre}
If \(n\) and \(q\) are relatively prime integers, then
\begin{multline*}
\sgn^+_n(q) = \begin{cases}
(-1)^{(q - 1)/2}, & n \equiv 0 \pmod4,    \\
\Legend q n,      & n \equiv 1, 3 \pmod4, \\
1,                & n \equiv 2 \pmod4
\end{cases}
\qand \\
\sgn^-_n(q) = \begin{cases}
\Legend n q, & n \equiv 0 \pmod4,    \\
\Legend q n, & n \equiv 1 \pmod4,    \\
1,           & n \equiv 2, 3 \pmod4.
\end{cases}
\end{multline*}
\end{prop}

\begin{rem}
By quadratic reciprocity, for \(q\) odd, we have the more nearly uniform expression
\[
\sgn^-_n(q) = \begin{cases}
\Legend n q, & n \equiv 0, 1 \pmod4, \\
1,           & n \equiv 2, 3 \pmod4.
\end{cases}
\]
\end{rem}

\begin{proof}
By Lemma \ref{prop:pre-Legendre}(\ref{prop:pre-Legendre:p-power}, \ref{prop:pre-Legendre:2-power}), the formul{\ae} are correct when \(n\) is a prime power.

By Lemma \hiref{prop:pre-Legendre}{CRT}, if \(m, n \in \Z\) are odd and relatively prime, then, on \((\Z/m n\Z)\mult\), we have that \(\sgn^+_{m n} = \sgn^+_m\sgn^+_n\), and that \(\sgn^-_{m n}\) is given by the following table:
\[\begin{array}{c|cc}
	& m \equiv 1 \pmod4        & m \equiv 3 \pmod4                 \\\hline
n \equiv 1 \pmod4
	& \sgn^-_m\sgn^-_n         & \sgn^+_n\sgn^-_m\sgn^-_n          \\
n \equiv 3 \pmod4
	& \sgn^+_m\sgn^-_m\sgn^-_n & \sgn^+_m\sgn^+_n\sgn^-_m\sgn^-_n.
\end{array}\]
We may then prove by induction on the number of distinct prime factors of \(n\) that
\begin{itemize}
\item the formula for \(\sgn^+_n\) is correct when \(n\) is odd,
\item \(\sgn^-_n = \sgn^+_n\) when \(n \equiv 1 \pmod4\),
and
\item \(\sgn^-_n = 1\) when \(n \equiv 3 \pmod4\).
\end{itemize}
These three facts together show that the formula for \(\sgn^-_n\) is correct when \(n\) is odd.

Again by Lemma \hiref{prop:pre-Legendre}{CRT}, if \(k \in \Z_{> 0}\) and \(n \in \Z\) is odd, then, on \((\Z/2^k n\Z)\mult\), we have that \(\sgn^+_{2^k n} = \sgn^+_{2^k}\), and that \(\sgn^-_{2^k n}\) is given by the following table:
\[\begin{array}{c|cc}
	& k = 1            & k > 1                             \\\hline
n \equiv 1 \pmod 4
	& \sgn^-_2         & \sgn^-_{2^k}\sgn^+_n              \\
n \equiv 3 \pmod 4
	& \sgn^+_2\sgn^-_2 & \sgn^+_{2^k}\sgn^-_{2^k}\sgn^+_n.
\end{array}\]
This shows that the formula for \(\sgn^+_n\) is correct for all \(n \in \Z\); and upon noting that, by what we have already shown and quadratic reciprocity  \cite{serre:arithmetic}*{Theorem I.I.3.3.6}, we have for \(k > 1\) and \(n \equiv 3 \pmod4\) that
\begin{multline*}
\sgn^+_{2^k}(q)\sgn^-_{2^k}(q)\sgn^+_n(q)
= (-1)^{(q - 1)/2}\Legend{2^k}q\Legend n q \\
= (-1)^{(q - 1)(n - 1)/4}\Legend{2^k}q\Legend n q
= \Legend{2^k n}q,
\end{multline*}
also that the formula for \(\sgn^-_n\) is correct for all \(n \in \Z\).
\end{proof}

\subsection{Root systems}
\label{sec:root-system}

For this section and \S\ref{sec:parabolic}, fix a root system \Root.  (All root systems are assumed to be reduced \bourbaki*{\S VI.1.4}.)  For \(\root \in \Root\), we write \mnotn{\refl_\root} for the reflection in the root \root; we will view this as a transformation of \(\R\Root\) or of its dual space, as appropriate.  Write \(W = \mnotn{W(\Root)}\) for the Weyl group of \Root, and \(A = \mnotn{\Aut(\Root)}\) for the group of automorphisms of \Root.  See \bourbaki*{Chapter VI} for details, especially \bourbaki*{\S VI.1.1} for basics on root systems.

If \Simple is a system of simple roots in \Root (\ie, a base of \Root, in the terminology of \bourbaki*{D\'efinition VI.1.5.2}), then we write \(\mnotn{\Aut(\Root, \Simple)} = \stab_{\Aut(\Root)}(\Simple)\).  Then
\begin{equation}
\label{eq:aut-sd}
A = W \rtimes \Aut(\Root, \Simple)
\end{equation}
\bourbaki*{Proposition VI.1.5.16}.  If we wish to emphasise the choice of \Simple, then we shall refer to \(\eqref{eq:aut-sd}_\Simple\) rather than just to \eqref{eq:aut-sd}.

Remember that we have already written \(\sgn_X\) for the sign character of a symmetric group (Definition \ref{defn:sgn-perm}).  We also write \(\sgn_\Root\) for a different character, defined below.

\begin{defn}
\label{defn:sgn-Root-A}
Let \Simple be a system of simple roots in \Root.  By \bourbaki*{Th\'eor\`eme IV.1.5.2(vii)}, there is a unique homomorphism \mnotn{\sgn_\Root} from \(A = W \rtimes \Aut(\Root, \Simple)\) to \sgen{-1} that is trivial on \(\Aut(\Root, \Simple)\), and sends \(\refl_\root\) to \(-1\) for each \(\root \in \Simple\).
\end{defn}

In fact, we can refine slightly the restriction of \(\sgn_\Root\) to \(W\).  This will be useful in Proposition \ref{prop:minus-sign}.

\begin{defn}
\label{defn:sgn-Root-sl}
Suppose that \Root is irreducible.  For \(\root \in \Root\), define
\[
\epsilon\textsub{\root, long} = \begin{cases}
-1, & \text{\Root is simply laced, or \(\root\) is long}       \\
1,  & \text{\Root is not simply laced, and \(\root\) is short}
\end{cases}
\]
(where, as usual, an irreducible root system is said to be simply laced if all roots have the same length; see \bourbaki*{Proposition VI.1.4.12(ii)}) and \(\epsilon\textsub{\root, short} = -\epsilon\textsub{\root, long}\).  Write \textasteriskcentered\ for `long' or `short'.  Fix a system \Simple of simple roots for \Root.  By \bourbaki*{VI.1.3, p.~148, 1)--11)}, especially \bourbaki*{VI.1.3, p.~148, 4)--7)}, we have that
\[
(\epsilon\textsub{\root, \textasteriskcentered}\epsilon\textsub{\otherroot, \textasteriskcentered})^{m(\root, \otherroot)} = 1 \qforall{\(\root, \otherroot \in \Simple\),}
\]
where, as usual, \(m(\root, \otherroot)\) is the order of \(\refl_\root\refl_\otherroot\); so, by \bourbaki*{D\'efinition IV.1.3.3 and Th\'eor\`eme VI.1.5.2(vii)}, there is a unique homomorphism
\mnonotn{\sgn\textsub{\Root, long}}\mnonotn{\sgn\textsub{\Root, short}}%
\map{\sgn\textsub{\Root, \textasteriskcentered}}W{\sgen{-1}} such that \(\sgn\textsub{\Root, \textasteriskcentered}(\refl_\root) = \epsilon\textsub{\root, \textasteriskcentered}\) for all \(\root \in \Simple\).

If \Root is reducible, then we define \(\sgn\textsub{\Root, \textasteriskcentered}\) in the obvious way, as the product over all irreducible components \(\Root'\) of \Root (with multiplicity) of the characters \(\sgn\textsub{\(\Root'\), \textasteriskcentered}\).
\end{defn}

\begin{rem}
By \bourbaki*{Th\`eor\`eme IV.1.5.2(i) and Proposition VI.1.5.15}, \(\sgn_\Root\) is trivial on any element of \(A\) that fixes a system of simple roots in \Root, and
\[
\sgn_\Root(\refl_\root) = -1
\qandq
\sgn\textsub{\Root, \textasteriskcentered}(\refl_\root) = \epsilon\textsub{\root, \textasteriskcentered}\qforall{\(\root \in \Root\).}
\]
In particular, \(\sgn_\Root\) and \(\sgn\textsub{\Root, \textasteriskcentered}\) do not depend on the choice of \Simple.

Note that \(\sgn_\Root = \sgn\textsub{\Root, long}\sgn\textsub{\Root, short}\).  If \Root is not simply laced, then \(\sgn\textsub{\Root, long}\) and \(\sgn\textsub{\dual\Root, short}\) are identified by the natural isomorphism \(W(\Root) \cong W(\dual\Root)\), where \dual\Root is the dual root system to \Root \bourbaki*{\S VI.1.1, p.~144}.
\end{rem}

\subsection{Parabolic subgroups of automorphism groups}
\label{sec:parabolic}

As in \S\ref{sec:root-system}, \Root is a root system with Weyl group \(W\) and automorphism group \(A\).

\begin{defn}
\label{defn:elliptic}
If \Simple is a system of simple roots in \Root, and \subSimple is a subset of \Simple, then we write
\[
\mnotn{\Root_\subSimple} = \Root \cap \Z\subSimple\textq,
\mnotn{W_\subSimple} = \gen{\refl_\root}{\root \in \subSimple}\textq,\andq
\mnotn{A_{\subSimple \subseteq \Simple}} = W_\subSimple \rtimes \stab_{\Aut(\Root, \Simple)}(\subSimple).
\]
A subsystem of \Root of the form \(\Root_\subSimple\) is called a \term{Levi subsystem}.  A subgroup of \(W\) of the form \(W_\subSimple\), or of \(A\) of the form \(A_{\subSimple \subseteq \Simple}\), is called \term{parabolic}; and an element \(w \in W\) (respectively, \(w \in A\)) is called \term{elliptic} if it belongs to no proper parabolic subgroup of \(W\) (respectively, of \(A\)).
\end{defn}

\begin{rem}
\toplabel{rem:parabolic}
We use the notation of Definition \ref{defn:elliptic}.
\begin{enumerate}
\item A subset \(\Root'\) of \Root is a Levi subsystem if and only if \(\R\Root' \cap \Root = \Root'\).  The `only if' direction is clear; for the `if' direction, see \bourbaki*{Proposition VI.1.7.24}.
\item\sublabel{refls} The restriction map \anonmap{W_\subSimple}{W(\Root_\subSimple)} is an isomorphism, and the sets \set{\root \in \Root}{\refl_\root \in W_\subSimple} and \set{\root \in \Root}{\refl_\root \in A_{\subSimple \subseteq \Simple}} both equal \(\Root_\subSimple\) \bourbaki*{Th\'eor\`eme IV.1.8.2(i) and Corollaire V.3.2}.
\item\sublabel{proper} By \pref{rem:parabolic:refls} (and \bourbaki*{\S VI.1.1 and Proposition VI.1.5.15}), a parabolic subgroup \(W_\subSimple\) of \(W\), or \(A_{\subSimple \subseteq \Simple}\) of \(A\), is proper if and only if \(\subSimple \ne \Simple\).
\item The intersection \(A_{\subSimple \subseteq \Simple} \cap W\) is \(W_\subSimple\).  In particular, by \pref{rem:parabolic:proper}, an element \(w \in W\) is elliptic (when considered as an element of \(W\)) if and only if it is elliptic when considered as an element of \(A\).
\end{enumerate}
\end{rem}

\begin{rem}
Parabolic subgroups of automorphism groups of root systems are less well behaved than parabolic subgroups of Coxeter groups.  For example, the group \(A_{\subSimple \subseteq \Simple}\) can depend on \Simple, not just on \subSimple; and the restriction map \anonmap{A_{\subSimple \subseteq \Simple}}{\Aut(\Root_\subSimple)} need not be injective or surjective.

For example, let \(\Root = \An[3]\), and use the notation of \bourbaki*{Planche VI.I}.  Write \(\root_0 = -(\root_1 + \root_2 + \root_3)\) for the lowest root of \Root (with respect to \Simple).  Put \(\Simple = \sset{\root_1, \root_2, \root_3}\) and \(\Simple' = \sset{\root_0, \root_1, \root_2}\), and let \(w_0\) be the non-trivial element of \(\Aut(\Root, \Simple)\).

Restriction furnishes an order-\(2\) covering \anonmap{A_{\sset{\root_2} \subseteq \Simple}}{\Aut(\Root_{\sset{\root_2}}) = W(\Root_{\sset{\root_2}})}, with kernel generated by \(w_0\), but an isomorphism \(A_{\sset{\root_2} \subseteq \Simple'} \cong \Aut(\Root_{\sset{\root_2}})\).

Since the unique element \(w_0\refl_{\root_1}\refl_{\root_2}\refl_{\root_3}\) of \(A\) that swaps \(\root_1\) and \(\root_2\) does not belong to \(A_{\sset{\root_1, \root_2} \subseteq \Simple}\) or \(A_{\sset{\root_1, \root_2} \subseteq \Simple'}\), neither surjects onto \(\Aut(\Root_{\sset{\root_1, \root_2}}) \cong W(\Root_{\sset{\root_1, \root_2}}) \times \Z/2\Z\).  (In fact, both map isomorphically onto \(W(\Root_{\sset{\root_1, \root_2}})\).)
\end{rem}

As for Weyl groups, there is a geometric interpretation of the notion of an elliptic automorphism of a Weyl group.

\begin{prop}
\label{prop:elliptic}
An element of \(A\) is elliptic if and only if it has no non-\(0\) fixed vectors in \(\R\Root\).
\end{prop}

\begin{proof}
For the `if' direction, suppose that
\begin{itemize}
\item \Simple is a system of simple roots for \Root,
\item \subSimple is a \emph{proper} subset of \Simple,
and
\item \(w \in A_{\subSimple \subseteq \Simple}\),
\end{itemize}
and let \(w = w' \rtimes w_0\) be the decomposition \(\eqref{eq:aut-sd}_\Simple\).  If \(\omega\) is an orbit of \(w_0\) on \(\Simple \setminus \subSimple\), then \(\sum_{\root \in \omega} \root\), hence also its projection \(\chi\) on the orthogonal complement \(\subSimple^\perp\), is \(w_0\)-fixed; and \(\chi\) belongs to \(\sum_{\root \in \omega} \root + \R\subSimple\), hence is non-\(0\).  Since \(w'\) is generated by reflections in elements of \subSimple, it fixes \(\subSimple^\perp\) (pointwise), hence also \(\chi\).

For the `only if' direction, suppose that \(\chi \in \R\Root\) is a non-\(0\) vector fixed by \(w\).  Choose a chamber \(C\) in \(\R\Root\), in the sense of \bourbaki*{\S VI.1.5}, whose closure contains \(\chi\) \bourbaki*{Th\'eor\`eme VI.1.5.2(ii)}.  Let \(w = w' \rtimes w_0\) be the decomposition \(\eqref{eq:aut-sd}_\Simple\), where \Simple is the system of simple roots defined by \(C\) \bourbaki*{D\'efinition VI.1.5.2}.  By \bourbaki*{Th\'eor\`eme VI.1.5.2(vi)}, we have that \(w_0\chi\) lies in the closure of \(C\).  Since \(w'w_0\chi = w\chi = \chi\), another application of \bourbaki*{Th\'eor\`eme VI.1.5.2(ii)} gives that \(w_0\chi = \chi\).  Now put \(\subSimple = \set{\root \in \Simple}{\pair\root\chi = 0}\).  Since \(w_0\chi = \chi\), also \(w_0\subSimple = \subSimple\).  By \bourbaki*{Proposition V.3.3.1}, \(w' \in W_\subSimple\).  That is, \(w \in A_{\subSimple \subseteq \Simple}\).  Since \(\chi \ne 0\), we have that \(\subSimple \ne \Simple\), so that \(A_{\subSimple \subseteq \Simple}\) is a proper parabolic subgroup of \(A\).
\end{proof}

Our goal in this paper is to compute the sign \Legend\sigma\theta (see Definition \ref{defn:root-sign}) in general.  In order to make this computation feasible, we need to reduce, in an appropriate sense, to elliptic automorphisms of Weyl groups.  The next two results allow us to do so.

\begin{prop}
\label{prop:parabolic-intersect}
The intersection of two parabolic subgroups of \(A\) is again a parabolic subgroup.
\end{prop}

\begin{proof}
We mimic the proof of \cite{geck-pfeiffer:characters}*{Theorem 2.1.12}.

Suppose that \Simple and \(\Simple'\) are two systems of simple roots in \Root.  By \bourbaki*{Remarque VI.1.5.4}, there exists \(u \in W\) so that \(u\Simple = \Simple'\).  Thus, it suffices to show that, for any two subsets \(\subSimple, \othersubSimple \subseteq \Simple\), the intersection \(A_{\subSimple \subseteq \Simple} \cap \Int(u)A_{\othersubSimple \subseteq \Simple}\) is a parabolic subgroup of \(A\).  We may, and do, further assume that \(u\) is of minimal length (with respect to \Simple; see \bourbaki*{Th\'eor\`eme VI.1.5.2(vii) and D\'efinition IV.1.1.1}) in \(W_\subSimple u W_\othersubSimple\).

Now suppose that \(v \in A_{\subSimple \subseteq \Simple}\) and \(w \in A_{\othersubSimple \subseteq \Simple}\) satisfy \(v = \Int(u)w\).  Let \(v = v' \rtimes v_0\) and \(w = w' \rtimes w_0\) be the decompositions \(\eqref{eq:aut-sd}_\Simple\), so that \(v' \in W_\subSimple\), \(w' \in W_\othersubSimple\), and \(v_0\) and \(w_0\) lie in \(\Aut(\Root, \Simple)\) and stabilise \subSimple and \othersubSimple (setwise), respectively.  Then comparing the components of \(v u = u w\) in the decomposition \(\eqref{eq:aut-sd}_\Simple\) shows that
\begin{equation}
\tag{$*$}
\label{prop:parabolic-intersect:eq:almost-conj}
v'v_0 u v_0\inv = u w'
\end{equation}
and \(v_0 = w_0\).

Since \(v_0 = w_0\) stabilises \Simple, \subSimple, and \othersubSimple (setwise), it preserves lengths (with respect to \Simple) and normalises \(W_\subSimple\) and \(W_\othersubSimple\); so \(v_0 u v_0\inv\) is again of minimal length in \(W_\subSimple\dotm v_0 u v_0\inv\dotm W_\othersubSimple\).

By \cite{geck-pfeiffer:characters}*{Propositions 2.1.1 and 2.1.7},
we have that \(v_0 u v_0\inv = u\), hence, by \cite{geck-pfeiffer:characters}*{Theorem 2.1.12} and \eqref{prop:parabolic-intersect:eq:almost-conj}, that \(v' = \Int(u)w' \in W_\subSimple \cap \Int(u)W_\othersubSimple = W_{\subSimple \cap u\othersubSimple}\).  Finally, \(v_0 = \Int(u)w_0 \in \stab_{\Aut(\Root, \Simple)}(\subSimple \cap u\othersubSimple)\), so
\[
v = v' \rtimes v_0 \in W_{\subSimple \cap u\othersubSimple} \rtimes \stab_{\Aut(\Root, \Simple)}(\subSimple \cap u\othersubSimple) = A_{\subSimple \cap u\othersubSimple \subseteq \Simple}.\qedhere
\]
\end{proof}

\begin{cor}
\toplabel{cor:Borel}
Suppose that \(w \in A\).
\begin{enumerate}
\item\sublabel{Borel} There is a unique smallest parabolic subgroup \(P\) of \(A\) containing \(w\).
\item\sublabel{elliptic} If \Simple is a system of simple roots in \Root, and \subSimple is a subset of \Simple, such that \(P = A_{\subSimple \subseteq \Simple}\), then the restriction of \(w\) to \(\Root_\subSimple\) is an elliptic element of \(\Aut(\Root_\subSimple)\).
\end{enumerate}
\end{cor}

\begin{proof}
\subpref{Borel} is an immediate consequence of Proposition \ref{prop:parabolic-intersect} (and the fact that \(A\) is finite).  To show \subpref{elliptic}, adopt the notation of the statement, and suppose that \(\subSimple'\) is a system of simple roots in \(\Root_\subSimple\), and \(\othersubSimple'\) a subset of \(\subSimple'\), such that the restriction \ol w of \(w\) to \(\Root_\subSimple\) lies in \(\Aut(\Root_\subSimple)_{\othersubSimple' \subseteq \subSimple'}\).  By \bourbaki*{Th\'eor\`eme IV.1.8.2(i)} and Remark \hiref{rem:parabolic}{refls}, there is \(v \in W_\subSimple\) such that \(v\subSimple = \subSimple'\).  Put \(\othersubSimple = v\inv\othersubSimple'\).

Let \(w = w' \rtimes w_0\) be the decomposition \(\eqref{eq:aut-sd}_\Simple\).  Then, with the obvious notation, \(\ol w = \ol w'\ol w_0\ol v\ol w_0\inv\ol v\inv \rtimes \ol v\ol w_0\ol v\inv\) is the decomposition \(\eqref{eq:aut-sd}_{\subSimple'}\).
\begin{itemize}
\item \(\ol w'\ol w_0\ol v\ol w_0\inv\ol v\inv \in W(\Root_\subSimple)_{\othersubSimple'}\), so \(\ol v\inv\ol w'\ol w_0\ol v\ol w_0\inv \in W(\Root_\subSimple)_\othersubSimple\).  Since \(W_\othersubSimple \subseteq W_\subSimple\) surjects onto \(W(\Root_\subSimple)_\othersubSimple\), we have by Remark \hiref{rem:parabolic}{refls} again that \(W_\othersubSimple\) is the full preimage of \(W(\Root_\subSimple)_\othersubSimple\) in \(W_\subSimple\); in particular, that \(v\inv w'w_0 v w_0\inv \in W_\othersubSimple\).
\item \(\ol v\ol w_0\ol v\inv\), hence also \(v w_0 v\inv\), stabilises \(\othersubSimple'\), so \(w_0\) stabilises \othersubSimple.
\end{itemize}
It follows that \(v\inv w v \in A_{\othersubSimple \subseteq \Simple}\), hence that \(w \in A_{\othersubSimple' \subseteq v\Simple}\).  By construction, \(A_{\subSimple \subseteq \Simple} \subseteq A_{\othersubSimple' \subseteq v\Simple}\).  By Remark \hiref{rem:parabolic}{refls} once more, we have that \(\Root_\subSimple \subseteq \Root_{\othersubSimple'}\).  The reverse containment being obvious, we have equality.  Since \(\dim \Root_\subSimple = \card\subSimple = \card{\subSimple'}\) and \(\dim \Root_{\othersubSimple'} = \card{\othersubSimple'}\), we have that \(\othersubSimple' = \subSimple'\), hence that \(\Aut(\Root_\subSimple)_{\othersubSimple' \subseteq \subSimple'} = \Aut(\Root_\subSimple)\).
\end{proof}

\begin{rem}
\label{rem:Borel}
With the notation of Corollary \ref{cor:Borel}, we have by Remark \hiref{rem:parabolic}{refls} that \(N_A(\sgen w) \subseteq N_A(P) \subseteq \stab_A(\Root_\subSimple)\).
\end{rem}

\section{Sign changes in Weyl groups}
\label{sec:signs}

\setcounter{equation}0

For the remainder of this paper, let \FF be a separably closed field.  For this section, let \(G\) be a connected, reductive, \FF-group, and \(T\) a maximal torus in \(G\).  Write \mnotn{\Lie(G)} and \(\Lie(T)\) for the Lie algebras of \(G\) and \(T\), respectively.  Put \(\Root = \mnotn{\Root(G, T)}\) \springer*{\S7.4.3, p.~125} and \(\mnotn{W(G, T)} = N_G(T)/T\), and write \(W = W(\Root)\) and \(A = \Aut(\Root)\).  Recall that \(W(G, T) \cong W\) \springer*{\S7.4.1}.  Write \(\Lie(G)_\root\) for the root subspace of \(\Lie(G)\) associated to a root \(\root \in \Root\), \ie, for the space on which \(T\) acts by \root.

\begin{defn}
\label{defn:auto}
An \term{algebraic automorphism} of \((G, T)\) is an algebraic automorphism of \(G\) that preserves \(T\).  A \term{twisted automorphism} of \((G, T)\) is the composition of an algebraic automorphism of \((G, T)\) with the action of some element of \(\Gal(\FF/\ff)\), where \(\FF/\ff\) is a Galois extension and both \(T\) and \(G\) are defined over \ff.
\end{defn}

Note that the group of algebraic automorphisms is a normal subgroup of the group of twisted automorphisms.

\subsection{The Tits group}
\label{sec:tits}

Since the definition of \(\Root_\theta\) is sensitive to the choice of \(\theta\), not just the automorphism \(w \ldef \ol\theta\) by which it acts on \Root, we might wonder if there is a canonical choice of \(\theta\), given \(w\).  This is not quite the case in general, but there is a next best thing.  Namely, Tits defines an `extended Weyl group' \(\dot W = \dot W(\Root)\) associated to \Root \tits*{D\'efinition 2.2}, which comes equipped with a surjection \anonmap{\dot W}W with kernel an elementary Abelian \(2\)-group \tits*{Th\'eor\`eme 2.5}, and shows that, if \(G\) is simply connected, then the choice of a system \Simple of simple roots and `realisation' of \Root in \(G\) (in the sense of \springer*{\S8.1.4}) furnishes an embedding of \(\dot W(\Root)\) in \(N_G(T)\) over \(W \cong W(G, T)\) \tits*{Th\'eor\`eme 4.4} (see also the definition of the category \TitsNpp in \tits*{\S3.2}).  In general, by factoring through the simply connected cover of the derived group of \(G\), the same data furnish an embedding of a quotient, which we call \mnotn{\dot W(G, T)} (and which is called \ol W in \springer*{Exercise 9.3.4(2)}), of \(\dot W(\Root)\) in \(N_G(T)\) over \(W\).  Note that the kernel of \anonmap{\dot W(G, T)}W is a quotient of the kernel of \anonmap{\dot W}W, hence still an elementary Abelian \(2\)-group.  Although we do not have a canonical embedding of \(\dot W(G, T)\) in \(G\), the various choices involved affect it only up to conjugacy in \(N_G(T)\).

There is quite a lot that can be said about the group \(\dot W(G, T)\), but all that we need to know is that it consists of elements that act by `semi-pinned' automorphisms, in the sense of the next definition.

\begin{defn}
\label{defn:semi-pinned}
An algebraic automorphism \(\theta\) of \((G, T)\) is called \term{semi-pinned} if, whenever \(\root \in \Root\) and \(e \in \Z\) are such that \(\theta^e\root = \root\), then \(\theta^e\) acts on \(\Lie(G)_\root\) by multiplication by \(+1\) or \(-1\).
\end{defn}

\begin{rem}
If \(\theta\) is semi-pinned and \(e \in \Z\) is such that \(\ol\theta^e = 1\), then \(\theta^{2e} = 1\).  In particular, a semi-pinned automorphism automatically has finite order.
\end{rem}

\begin{lem}
\label{lem:W+-}
Every element of \(\dot W(G, T)\) induces a semi-pinned automorphism of \(G\) (by conjugation).
\end{lem}

\begin{proof}
Since \(\dot W(G, T)\) is a group, it suffices to show that, if \(\dot w \in \dot W(G, T)\) fixes \(\root \in \Root\), then \(\dot w\) acts on \(\Lie(G)_\root\) by multiplication by \(+1\) or \(-1\).

Remember that the construction of \(\dot W(G, T)\) involves a choice of a system \Simple of simple roots for \Root.  By \bourbaki*{Proposition VI.1.5.15}, there is \(v \in W\) so that \(v\root \in \Simple\).  By \springer*{Proposition 9.3.5}, in the notation of \springer*{\S9.3.3}, the element \(\dot u \ldef \phi(v w v\inv) \in \dot W(G, T)\) fixes \(\Lie(G)_{v\root}\) (pointwise); and, by \springer*{Exercise 9.3.4(2)}, if \(\dot v\) is any lift of \(v\) to \(\dot W(G, T)\), then \(t = \dot u\inv\dot v\dot w\dot v\inv\) is a \(2\)-torsion element of \(T\).  Thus \(t\), hence also \(\dot v\dot w\dot v\inv = \dot u t\), acts on \(\Lie(G)_{v\root}\) by multiplication by the scalar \(\alpha(t) \in \sset{\pm1}\); so \(\dot w\) acts on \(\Lie(G)\) by the same scalar.
\end{proof}

Thus, the extended Weyl group provides a source of semi-pinned automorphisms lifting all elements of the Weyl group.  We would like to be able to lift arbitrary automorphisms of the root system, whether or not they lie in the Weyl group, and we can nearly always do so.

\begin{prop}
\label{prop:A+-}
Suppose that \(w \in A\), and
\begin{itemize}
\item \(w \in W\),
or
\item \(G\) is simple, adjoint, or simply connected.
\end{itemize}
Then there is a semi-pinned automorphism \(\theta\) of \((G, T)\) such that \(\ol\theta = w\).
\end{prop}

\begin{rem}
For our purposes, the conditions of the proposition are not as restrictive as they seem; for, if \(G\) is arbitrary, then we may simply pass to its adjoint quotient, which does not affect its root system \Root.
\end{rem}

\begin{proof}
We have already handled the case where \(w \in W\) in Lemma \ref{lem:W+-}.

Suppose that we have proven the result for all simple groups.  Note that we may write \(w = w'w_0\), where, for each irreducible component \(\Root'\) of \Root,
\begin{itemize}
\item the element \(w_0\) stabilises \(\Root'\) (setwise),
and
\item the element \(w'\) fixes \(\Root'\) (pointwise), or carries it to a different irreducible component of \Root.
\end{itemize}
Now suppose that \(G\) is adjoint or simply connected, so that it is the product of its simple subgroups, all of which are adjoint, or all of which are simply connected.  (See \springer*{Theorem 8.1.5}.)  If \(G'\) and \(G''\) are simple subgroups of \(G\), corresponding to irreducible components \(\Root'\) and \(\Root''\) of \Root such that \(\Root'' = w'\Root'\), then \(G'\) and \(G''\) are isomorphic \springer*{Theorem 9.6.2}.  We piece together all such isomorphisms to obtain an automorphism \(\theta'\) of \(G\).  Next we piece together, for each simple subgroup \(G'\) of \(G\), a semi-pinned lifting of \(w_0\) to an automorphism of \(G'\), and call the result \(\theta_0\).  Then \(\theta = \theta'\theta_0\) is a semi-pinned automorphism of \(G\) such that \(\ol\theta = w\).

Thus, it suffices to handle the case where \(G\) is simple.  By Lemma \ref{lem:W+-}, it suffices to show that we may find
\begin{itemize}
\item a connected, reductive \FF-group \wtilde J,
\item a maximal torus \wtilde T in \wtilde J,
\item a subgroup \wtilde G of \wtilde J containing \wtilde T,
and
\item a central subgroup \wtilde Z of the derived subgroup \(\mc D\wtilde G\) of \wtilde G,
\end{itemize}
such that
\begin{itemize}
\item there is an isomorphism \(\mc D\wtilde G/\wtilde Z \cong G\) that restricts to an isomorphism \(\mc D\wtilde G \cap \wtilde T/\wtilde Z \cong T\), so that \Root may be naturally identified with the subset \(\Root(\mc D\wtilde G, \mc D\wtilde G \cap \wtilde T)\) of \(\Root(\wtilde J, \wtilde T)\),
and
\item \(\stab_{W(\wtilde J, \wtilde T)}(\Root)\) contains a set of representatives for \(A/W\).
\end{itemize}

Fortunately, there are only a few cases where \(A \ne W\).  In each case, we find a root system \wtilde\Root of which \Root is a Levi subsystem (in the sense of Definition \ref{defn:elliptic}),
and take \wtilde J to be the derived group of the simply connected group with root system \wtilde\Root, and \wtilde G to be an appropriate Levi subgroup.
We will label root systems as in \bourbaki*{Planches VI.IV--VI}, and use the notation \(L_I\) of \springer*{\S15.4, p.~264} for Levi components of `standard' parabolic subgroups.

If \Root is of type \An with \(n > 2\), then we take \(\wtilde\Root = \Dn\) and \(L = L_{\sset{\root_1, \dotsc, \root_{n - 1}}}\).  The element
\[
\prod_{i = 1}^{\rup{n/2} - 1} \refl_{\root_1 + \dotsb + \root_{n - i - 1} + 2\root_{n - i} + \dotsb + 2\root_{n - 2} + \root_{n - 1} + \root_n} 
\in N_{W(\Dn)}(W)
\]
acts as an outer automorphism of \Root.

If \Root is of type \Dn with \(n > 4\), then we take \(\wtilde\Root = \Dn[n + 1]\) and \(L = L_{\sset{\root_2, \dotsc, \root_{n + 1}}}\).  The element
\[
\refl_{\root_1}\refl_{\root_1 + 2\root_2 + \dotsb + 2\root_{n - 1} + \root_n + \root_{n + 1}} 
\in N_{W(\Dn[n + 1])}(W)
\]
acts as an outer automorphism of \Root.

If \Root is of type \Dn[4], then we take \(\wtilde\Root = \En[6]\) and \(L = L_{\sset{\root_2, \dotsc, \root_5}}\).  The elements
\[
\refl_{\Eroot6 0 1 0 1 1 1}
\refl_{\Eroot6 0 0 1 1 1 1}, 
\refl_{\Eroot6 1 1 1 1 0 0}\refl_{\Eroot6 1 0 1 1 1 0} \in N_{W(\En[6])}(W)
\]
act as outer automorphisms of \Root that are not congruent modulo inner automorphisms.

If \Root is of type \En[6], then we take \(\wtilde\Root = \En[7]\) and \(L = L_{\sset{\root_1, \dotsc, \root_6}}\).  The element
\[
\refl_{\Eroot7 0 1 1 2 2 2 1}
\refl_{\Eroot7 1 1 2 2 1 1 1}\refl_{\Eroot7 1 1 1 2 2 1 1} \in N_{W(\En[7])}(W)
\]
acts as an outer automorphism of \Root.
\end{proof}

\subsection{A set of roots attached to an algebraic automorphism}
\label{sec:R_theta}

We continue with the notation established at the beginning of \S\ref{sec:signs}.  Until \S\ref{sec:root-sign}, let \(\theta\) be an algebraic automorphism of \((G, T)\) (Definition \ref{defn:auto}).

Loosely speaking, one may motivate the next definition by thinking of the case in which \(C_G(T^\theta) = T\).  Then restriction defines a bijection of \(\sgen\theta\bslash\Root_\theta\) with \(\Root((G^\theta)\conn, (T^\theta)\conn)\).

\begin{defn}
\label{defn:R_theta}
Put
\[
\mnotn{\Root_\theta} = \sett{\root \in \Root}{\(\theta^e\) fixes \(\Lie(G)_\root\) pointwise whenever \(\theta^e\root = \root\)}.
\]
\end{defn}

It is easy to construct examples where \(\Root_\theta\) depends on \(\theta\), not just on \ol\theta; but the next two results show that this cannot happen if \ol\theta is elliptic.

\begin{lem}
\label{lem:elliptic-cohom}
If \(w \in \Aut(\Root)\) is elliptic and \(G\) is semisimple, then \(T^w\) is finite and \(T = (1 - w)T\).
\end{lem}

\begin{proof}
By Proposition \ref{prop:elliptic}, the element \(w\) has no non-\(0\) fixed vectors in \(\R\Root = \bX^*(T) \otimes_\Z \R\), hence also none in the dual space \(\bX_*(T) \otimes_\Z \R\); so
\[
\bX_*(T^w) \otimes_\Z \R = \bX_*(T)^w \otimes_\Z \R \cong \bigl(\bX_*(T) \otimes_\Z \R\bigr)^w = 0,
\]
whence \(T^w\) contains no non-trivial subtori.  By \springer*{Corollary 3.2.7(ii)}, this means that \(T^w\) is finite, hence that \(\dim(T^w) = 0\).  Since \map{1 - w}{T/T^w}{(1 - w)T} is an isomorphism, this means that \(\dim T = \dim (1 - w)T\).  Since \(T\) is connected, it follows that \(T = (1 - w)T\).
\end{proof}

\begin{cor}
\label{cor:elliptic-cohom}
If \ol\theta is elliptic, and \(\theta'\) is an algebraic automorphism of \((G, T)\) such that \(\ol\theta' = \ol\theta\), then \(\theta' = \Int(s)\theta\Int(s)\inv\) for some \(s \in T(\FF)\).
\end{cor}

\begin{proof}
We may write \(\theta' = \Int(t)\theta\) for some \(t \in T(\FF)\) \cite{springer:corvallis}*{Proposition 2.5(ii)}.  Since \(\Int(t) \in \Int(T)(\FF) = \bigl((1 - w)\Int(T)\bigr)(\FF) = (1 - w)(\Int(T)(\FF))\) by Lemma \ref{lem:elliptic-cohom}, the result follows.
\end{proof}

\begin{defn}
\label{defn:Rw}
If \(w \in A\) is elliptic, then we put \(\mnotn{\Root_w} = \Root_\theta\), where \(\theta\) is any algebraic automorphism of \((G, T)\) such that \(\ol\theta = w\).  By Corollary \ref{cor:elliptic-cohom}, the set \(\Root_w\) does not depend on the choice of \(\theta\).
\end{defn}

\begin{rem}
\label{rem:not-root}
\newcommand\diag{\operatorname{diag}}
The set \(\Root_\theta\) need not be a root system.  For example, put \(G = \GL_6\), and let \(T\) be the subgroup of diagonal matrices, with roots
\[
\mapto{\root_1}{\diag(t_0, \dotsc, t_5)}{t_0 t_1\inv}
\qandq
\mapto{\root_2}{\diag(t_0, \dotsc, t_5)}{t_1 t_2\inv}.
\]
Put \(\otherroot = \refl_{\root_1}(\root_2)\).  Let \(w\) be the permutation matrix corresponding to the permutation \(\cycle(1 3)\cycle(2 4 5)\), and put \(\theta = \Int\bigl(w\dotm\diag(-1, 1^{+5})\bigr)\).  Then
	\begin{itemize}
	\item \(\theta^e\root_1 = \root_1\) if and only if \(2 \mid e\), and \(\theta^2 = \Int(w^2)\) acts trivially on \(\Lie(G)_{\root_1}\), so \(\root_1 \in \Root(G, T)_\theta\);
	\item \(\theta^e\root_2 = \root_2\) if and only if \(6 \mid e\), and \(\theta^6 = 1\) acts trivially on \(\Lie(G)_{\root_2}\), so \(\root_2 \in \Root(G, T)_\theta\);
	but
	\item \(\theta^3\otherroot = \otherroot\) and \(\theta^3 = \Int\bigl(w^3\dotm\diag(-1, 1^{+5})\bigr)\) acts on \(\Lie(G)_\otherroot\) by multiplication by \(-1\), so \(\otherroot \not\in \Root(G, T)_\theta\).
	\end{itemize}
\end{rem}

Despite Remark \ref{rem:not-root}, it is clear that \(-\Root_\theta = \Root_\theta\), and, further, that \(\Root_\theta \subseteq \Root_{\theta^n}\) for all \(n \in \Z\).  We can do better in some cases.

\begin{lem}
\label{lem:odd-power}
If \(\theta\) is semi-pinned (see Definition \ref{defn:semi-pinned}) and \(n\) is odd, then \(\Root_\theta = \Root_{\theta^n}\).
\end{lem}

\begin{proof}
Suppose that \(\root \in \Root\), and \(e \in \Z\) is such that \(\theta^e\root = \root\).  Then \(\theta^e\) acts on \(\Lie(G)_\root\) by multiplication by \(\varepsilon \in \sset{\pm1}\); so \((\theta^n)^e\) acts on \(\Lie(G)_\root\) by multiplication by \(\varepsilon^n = \varepsilon\).  If \(\root \in \Root_{\theta^n}\), then \((\theta^n)^e\) must fix \(\Lie(G)_\root\) pointwise, so that \(\varepsilon = 1\).  Since \(e\) was any integer such that \(\theta^e\root = \root\), this means that \(\root \in \Root_\theta\).
\end{proof}

\subsection{A sign associated to a pair of automorphisms}
\label{sec:root-sign}

We continue with the notation established at the beginning of \S\ref{sec:signs}.  As in \S\ref{sec:root-sign}, we let \(\theta\) be an algebraic automorphism of \((G, T)\).  Now let also \(\sigma\) be a twisted automorphism of \((G, T)\) (Definition \ref{defn:auto}) that normalises \(\sgen\theta\), so that \(\ol\sigma\) stabilises \(\Root_\theta\) (setwise).

As described in \S\ref{sec:intro}, we are interested in the number (or at least its parity) of orbits of \sgen{\sigma, \theta} on \Root, but it turns out to be better to use Lemma \ref{lem:perm-sign} to re-phrase this computation in terms of the sign of a certain permutation---because the sign is multiplicative in \(\sigma\), whereas the count of the number of orbits need not be.

\begin{defn}
\label{defn:root-sign}
We write \(\mnotn{\displaystyle\Legend\sigma\theta_{\!\Root}}\) or \(\Legend\sigma\theta_\Root\) for the sign of the permutation of \(\sgen\theta\bslash\Root_\theta\) induced by \(\sigma\).  We may omit the subscript \(\Root\) if it is clear from the context.  If \(w \ldef \ol\theta\) is elliptic, then (by Corollary \ref{cor:elliptic-cohom}) the sign \Legend\sigma\theta depends only on \(w\), so we may, and do, denote it by
\mnonotn{\displaystyle\Legend\sigma w_{\!\Root}}%
\Legend\sigma w.
\end{defn}

\begin{lem}
\label{lem:plus-sign}
\(\Legend\sigma1 = \sgn_\Root(\ol\sigma)\).
\end{lem}

\begin{rem}
One may view Lemma \ref{lem:plus-sign} as saying that the two meanings of \(\sgn_\Root\), as a character of \(\Sgp_\Root\) (Definition \ref{defn:sgn-perm}) and of \(\Aut(\Root)\) (Definition \ref{defn:sgn-Root-A}), agree on \(\Aut(\Root)\).
\end{rem}

\begin{proof}
Obviously, \(\Root_1 = \Root\).  Let \Simple be a system of simple roots in \Root.  Put \(v = \ol\sigma\).

It suffices to show the result in case \(v\) stabilises \Simple, or \(v = \refl_\root\) for some \(\root \in \Simple\).  In the former case, \(v\) has no symmetric orbits on \Root, so that \(\Legend\sigma1 = 1 = \sgn_\Root(v)\).  In the latter case, by \bourbaki*{Corollaire VI.1.6.1}, \(v\) has exactly one symmetric orbit on \Root (namely, \sset{\pm\root}), so \(\Legend\sigma1 = -1 = \sgn_\Root(v)\).
\end{proof}

\begin{rem}
\label{rem:plus-sign}
Since \card\Root is even, we have by Lemma \ref{lem:perm-sign} that \(\Legend\sigma1 = (-1)^{\card{\sgen\sigma\bslash\Root}}\).
\end{rem}

\begin{rem}
\label{rem:reduce-irreducible}
The computation of the symbol \(\Legend\sigma\theta_\Root\) in general may be reduced to its computation in case \Root is irreducible, \ie, \(G\) is simple.

To see this, put \(\Gamma = \sgen{\sigma, \theta}\) and \(N = \sgen\theta\).

If \(\Root'\) is an irreducible component of \Root, then write \(e_{\Root'}\) for the least positive integer \(e\) such that \(\theta^e\) stabilises \(\Root'\) (setwise), and \(f_{\Root'}\) for the least positive integer \(f\) such that \(\sigma^f\) stabilises \(N\dota\Root'\).  Put \(\theta_{\Root'} = \theta^{e_{\Root'}}\), and write \(\sigma_{\Root'}\) for any element of \(N\dotm\sigma^{f_{\Root'}}\) that stabilises \(\Root'\) (setwise).  Put \(N_{\Root'} = \sgen{\theta_{\Root'}} = \stab_N(\Root')\) and \(\Gamma_{\Root'} = \sgen{\sigma_{\Root'}, \theta_{\Root'}} = \stab_\Gamma(\Root')\).  Then \(\theta_{\Root'}\) preserves the group \(G'\) generated by \(T\) and the root subgroups of \(G\) corresponding to roots in \(\Root'\), which has root system \(\Root'\); and we have that \(\Root' \cap \Root_\theta = \Root'_{\theta_{\Root'}}\).  Note that \(e_{\Root'}\) (hence \(\theta_{\Root'}\)) and \(f_{\Root'}\) depend only on the \sgen{\sigma, \theta}-orbit of \(\Root'\).
We have that \(N\bslash N\dota(\Root' \cap \Root_\theta)\) is naturally in bijection with \(N_{\Root'}\bslash\Root'_{\theta_{\Root'}}\), and \(\Gamma\bslash\Gamma\dota(\Root' \cap \Root_\theta)\) with \(\Gamma_{\Root'}\bslash\Root'_{\theta_{\Root'}}\).  Finally, the set of irreducible components of \(\Gamma\dota\Root'\) falls into precisely \(e_{\Root'}\) orbits under \(N\).

By Lemma \ref{lem:perm-sign}, we have that
\[
\Legend\sigma\theta_{\!\Root} = (-1)^{\card{N\bslash\Root} - \card{\Gamma\bslash\Root}}.
\]
The above discussion shows that we may write
\begin{align*}
\Legend\sigma\theta_{\!\Root}
& {}= \prod (-1)^{e_{\Root'}\dotm\card{N_{\Root'}\bslash\Root'} - \card{\Gamma_{\Root'}\bslash\Root'}} \\
& {}= \prod (-1)^{\card{N_{\Root'}\bslash\Root'} - \card{\Gamma_{\Root'}\bslash\Root'}}(-1)^{(e_{\Root'} - 1)\dotm\card{\sgen{\theta_{\Root'}}\bslash\Root'}},
\end{align*}
where the product runs over the \(\Gamma\)-orbits of irreducible components \(\Root'\) of \Root.  By Remark \ref{rem:plus-sign} (and Lemma \ref{lem:perm-sign} again), this may be re-written as
\[
\Legend\sigma\theta_{\!\Root}
= \prod \Legend{\sigma_{\Root'}}{\theta_{\Root'}}_{\!\Root'}\dotm\Legend{\theta_{\Root'}}1_{\!\Root'}^{e_{\Root'} - 1}.
\]
\end{rem}

\begin{rem}
\label{rem:irrational}
It is clear that \Legend\sigma\theta depends only on \(v \ldef \ol\sigma\) (and \(\theta\)), so we will sometimes denote it by
\mnonotn{\displaystyle\Legend v\theta_{\!\Root}}%
\Legend v\theta.

\newcommand\diag{\operatorname{diag}}
Note, however, that the set \(\Root_\theta\) need not be stable under \(N_A(\smashed\sgen{\ol\theta})\), or even \(C_W(\ol\theta)\).  This puts some constraints on the possible `numerators' \(v\) in a symbol \Legend v\theta; but see Proposition \ref{prop:fix-lift} below.

For example, put \(G = \GL_3\), and let \(T\) be the subgroup of diagonal matrices, with roots
\[
\map{\root_1}{\diag(t_0, t_1, t_2)}{t_0 t_1\inv}
\qandq
\map{\root_2}{\diag(t_0, t_1, t_2)}{t_1 t_2\inv}.
\]
Put \(\theta = \Int(\diag(1, 1, -1))\).  Then \(\refl_{\root_2}\) commutes with \(\ol\theta = 1\), but does not stabilise \(\Root_\theta = \sset{\pm\root_1}\); so there is no abstract automorphism \(\sigma\) such that \(\sigma\) normalises \sgen\theta and \(\ol\sigma = \refl_{\root_2}\).

Further, \Legend\sigma\theta may depend on \(\theta\) (and \ol\sigma), not just on \ol\theta.  For example, put \(G = \GL_2\), and let \(T\) be the subgroup of diagonal matrices.  Let \(\sigma\) be the inverse-transpose map on \(G\),
so that \(\ol\sigma = -1\).  If \(\theta = 1\) and \(\theta' = \Int(\diag(1, -1))\), then \(\ol\theta = \ol\theta'\); but \(\Root_\theta = \Root\), so \(\Legend{-1}\theta = -1\), and \(\Root_{\theta'} = \emptyset\), so \(\Legend{-1}{\theta'} = 1\).  If \(w \ldef \ol\theta\) is elliptic, then this issue does not arise, so that we may write
\mnonotn{\displaystyle\Legend v w_{\!\Root}}%
\Legend v w in place of \Legend\sigma\theta.
\end{rem}

The next two results show that the possible obstruction to defining \Legend v\theta mentioned in Remark \ref{rem:irrational} does not arise if \ol\theta is elliptic.

\begin{lem}
\label{lem:rational}
An \ff-structure on \(T\) may be extended to an \ff-structure on \(G\) if and only if the action of \(\Aut(\FF)\) on \(\bX^*(T)\) is by automorphisms of \Root.
\end{lem}

\begin{proof}
The `only if' direction is obvious.

For the `if' direction, note that, if \(\ff\sep\) is the separable closure of \ff in \FF, then \(T\) is \(\ff\sep\)-split for any \ff-structure, so that \(\Aut(\FF/\ff\sep)\) acts trivially on \(\bX^*(T)\) \springer*{Proposition 13.1.1(ii) and 13.2.2(i)}; so we may, and do, assume that \(\FF = \ff\sep\).  We shall write \(G_\FF\) and \(T_\FF\) instead of \(G\) and \(T\) when we wish to emphasise that we are viewing these objects as groups over \FF.

Now let
\begin{itemize}
\item \map\varphi{\Gal(\FF/\ff)}A be the action map for the chosen \ff-structure on \(T\),
\item \Simple be a system of simple roots in \Root,
and 
\item \(\varphi_0\) be the composition of \(\varphi\) with the projection \anonmap A{\Aut(\Root, \Simple)} along \(\eqref{eq:aut-sd}_\Simple\).
\end{itemize}
By \springer*{\S16.4.7}, there is a (quasisplit) \ff-structure \((G\qsform, T\qsform)\) on \((G_\FF, T_\FF)\) for which \(\varphi_0\) is the action map on \(\bX^*(T)\).  By \cite{raghunathan:tori}*{Theorem 1.1}, there is a \(1\)-cocycle \(c\) on \(\Gal(\FF/\ff)\) with values in \(N_{G\qsform}(T\qsform)\) such that, for all \(\sigma \in \Gal(\FF/\ff)\), the element \(c_\sigma\) is a preimage in \(N_{G\qsform}(T\qsform)\) of \(\varphi(\sigma)\varphi_0(\sigma)\inv \in W \cong W(G\qsform, T\qsform)\).  The cocycle \(c\) naturally gives rise, \via the interior action, to cocycles with values in \(\Aut(G\qsform)\) and \(\Aut(T\qsform)\), so we may construct the corresponding twists \(G\) and \(T\) of \(G\qsform\) and \(T\qsform\), respectively \springer*{\S11.3.3}.  We have that \(T\) is \ff-isomorphic to the chosen \ff-structure, so that \(G\) is the desired \ff-structure on \(G_\FF\).
\end{proof}

\begin{prop}
\toplabel{prop:fix-lift}
Suppose that
\begin{itemize}
\item \(\FF/\ff\) is a procyclic Galois extension,
\item \(\Frob\) is a progenerator of \(\Gal(\FF/\ff)\),
\item \(w \in A\) is elliptic,
\item \(v \in A\) and \(q \in \Z\) satisfy \(v w v\inv = w^q\),
and
\item \(\theta \in \Aut(G, T)\) satisfies \(\ol\theta = w\).
\end{itemize}
Then there is an \ff-structure on \((G, T)\) such that
\begin{itemize}
\item \(\Frob\) acts on \(\bX^*(T)\) by \(v\),
and
\item \(\Frob\theta\Frob\inv = \theta^q\).
\end{itemize}
\end{prop}

\begin{proof}
As in the proof of Lemma \ref{lem:rational}, we shall write \(G_\FF\) and \(T_\FF\) instead of \(G\) and \(T\) when we wish to emphasise that we are viewing these objects as groups over \FF.

By Lemma \ref{lem:rational}, there is \emph{some} \ff-structure \((G_0, T_0)\) on \((G_\FF, T_\FF)\) for which \(\Frob\) acts on \(\bX^*(T)\) by \(v\).  For emphasis, we shall denote the action of \(\Frob\) according to this structure by \(\Frob_0\).

By Corollary \ref{cor:elliptic-cohom}, there is an element \(s \in T_0(\FF)\) such that
\begin{equation}
\tag{$*$}
\sublabel{eq:equivalent}
\Frob_0\theta\Frob_0\inv = \Int(s)\theta^q\Int(s)\inv.
\end{equation}
Choose a positive integer \(d\) so that \(\theta\) commutes with, and \(s\) is fixed by, \(\Frob_0^d\).  Then conjugating \subeqref{eq:equivalent} repeatedly by \(\Frob_0\) gives
\[
\theta = \Int(t)\theta\Int(t)\inv = \Int((1 - w)t)\theta,
\]
where \(t = \prod_{i = 0}^{d - 1} \Frob_0^i s \in T_0(\ff)\).  Then \((1 - w)t = 1\), \ie, \(t \in T_0(\ff)^w = T_0^w(\ff)\).  By Lemma \ref{lem:elliptic-cohom}, the group \(\Int(T_0)^w(\FF)\) has finite order, say \(N\); so
\[
\Int\Bigl(\prod_{i = 0}^{N d - 1} \Frob_0^i s\Bigr) = \Int\Bigl(\prod_{j = 0}^{N - 1} \Frob_0^j t\Bigr) = \Int(t)^N = 1.
\]
That is, there is a \(1\)-cocycle \(c\) on \(\Gal(\FF/\ff)\) with values in \(\Int(T_0) \subseteq \Aut(G_0)\) such that \(c_{\Frob} = \Int(s)\inv\).  Let \((G, T)\) be the twist of \((G_0, T_0)\) by \(c\).
\end{proof}

We have already computed \Legend\cdot1 (see Lemma \ref{lem:plus-sign}).  The next result nearly computes \Legend\cdot{-1}, and Remark \ref{rem:minus-sign} finishes the job.

\begin{prop}
\label{prop:minus-sign}
If \Root is irreducible and \(v \in W\), then
\[
\Legend v{-1} = \sgn\textsub{\Root, long}(v)^g\sgn\textsub{\Root, short}(v)^{\dual g},
\]
where \(\sgn\textsub{\Root, \textasteriskcentered}\) is as in Definition \ref{defn:sgn-Root-sl}, and \(g\) and \dual g are the dual Coxeter numbers \cite{suter:coxeter}*{(1), p.~148} of \Root and its dual root system, respectively.
\end{prop}

\begin{proof}
By \cite{bourbaki:lie-gp+lie-alg_7-9}*{Proposition VIII.4.4.5}, we have that \(\Root_{-1} = \Root\).  It suffices to assume that \(v = \refl_\root\) for some \(\root \in \Root\).

Choose a system \Simple of simple roots that includes \root \bourbaki*{Proposition VI.1.5.15}.  Since \(\refl_\root\) permutes the (\Simple-)positive roots other than \root \bourbaki*{Corollaire 1 to Proposition VI.1.6.17}, the fixed points of \(\refl_\root\) on \(\dot\Root \ldef \sgen{-1}\bslash\Root\) are precisely \sset{\pm\alpha} and each \sset{\pm\otherroot}, where \(\otherroot \in \Root\) is a root perpendicular to \root.  By \cite{suter:coxeter}*{Proposition 1}, if \root is long (respectively, short), then there are \(4g - 6\) (respectively, \(4\dual g - 6\)) roots \emph{not} perpendicular to it, hence \(2g - 4\) (respectively, \(2\dual g - 4\)) non-fixed points of \(\refl_\root\) on \(\dot\Root\).  Since \(\refl_\root\) acts as an involution of \(\dot\Root\), it must therefore be a product of \(g - 2\) (respectively, \(\dual g - 2\)) transpositions, so
\[
\Legend v{-1} = (-1)^{g - 2} = \sgn\textsub{\Root, long}(v)^g = \sgn\textsub{\Root, long}(v)^g\sgn\textsub{\Root, short}(v)^{\dual g}
\]
(respectively, \(\Legend v{-1} = \sgn\textsub{\Root, short}(v)^{\dual g} = \sgn\textsub{\Root, long}(v)^g\sgn\textsub{\Root, short}(v)^{\dual g}\)).
\end{proof}

\begin{rem}
\label{rem:minus-sign}
Since \(\Legend v{-1} = 1\) for \(v = -1\), Proposition \ref{prop:minus-sign} completely determines \(\Legend\cdot{-1}\) except for \Root of type \Dn with \(n\) even.  If \(v\) is an involutive automorphism of \Dn that stabilises (setwise) a system of simple roots, then, after relabelling if necessary, we have that \(v\) acts on \(\R\Dn = \bigoplus_{i \in \Z/n\Z} \R\basis_i\) by fixing \(\basis_i\) for \(i \ne 0\), and negating \(\basis_0\).  Then \(v\) acts on \(\sgen{-1}\bslash\Root\) by the product of the \(n - 1\) transpositions swapping \(\sgen{-1}\dota(\basis_0 - \basis_i)\) and \(\sgen{-1}\dota(\basis_0 + \basis_i)\) for \(i \ne 0\); so
\[
\Legend v{-1} = (-1)^{n - 1}.
\]
\end{rem}

\begin{rem}
\label{rem:dual-Coxeter}
In light of Proposition \ref{prop:minus-sign}, it is helpful to have the following chart of dual Coxeter numbers.  Note that, for a simply laced root system, the dual Coxeter number is just the Coxeter number.
\[\begin{array}{c|c}
\text{type of \Root} & \text{dual Coxeter number} \\ \hline
\An                  & n                          \\
\Bn                  & 2n - 1                     \\
\Cn                  & n + 1                      \\
\Dn                  & 2(n - 1)                   \\
\Gn                  & 4                          \\
\Fn                  & 9                          \\
\En[6]               & 12                         \\
\En[7]               & 18                         \\
\En[8]               & 30
\end{array}\]
In particular, by Remark \ref{rem:minus-sign}, the kernel of \(\Legend\cdot{-1}\) is
\begin{itemize}
\item \(A\) if \Root is of type \An with \(n\) even, \Dn with \(n\) odd, \Gn, or \En with \(n \in \sset{6, 7, 8}\),
\item \(\ker \sgn_{\An} \times \sgen{-1}\) if \(\Root = \An\) with \(n\) odd,
\item \(W(\Dn)\) if \Root is of type \Bn or \Cn, or \(\Root = \Dn\) with \(n > 4\) even,
\item the unique index-\(2\) subgroup of \(A\) containing \(W\) if \(\Root = \Dn[4]\),
and
\item \(\ker \sgn_{\Fn}\) if \(\Root = \Fn\).
\end{itemize}
\end{rem}

Recall that our goal is to reduce the computation of \Legend\sigma\theta in general to a manageable case-by-case computation.  The next result allows us to reduce some computations to the case where \ol\theta has \(2\)-power order.  Note that Lemma \ref{lem:odd-power} gives conditions under which the hypothesis on \(\Root_\theta\) is satisfied.

\begin{lem}
\toplabel{lem:reduce-2-power}
If \(n\) is an odd integer and \(\Root_\theta = \Root_{\theta^n}\), then \(\Legend\sigma\theta = \Legend\sigma{\theta^n}\).
\end{lem}

\begin{proof}
Put \(w = \ol\theta\).

For any \(v \in N_A(\sgen w) \cap \stab_A(\Root_\theta)\) and any \(\omega \in \sgen{v, w}\bslash\Root_\theta\), we have that \(\sgen w/\sgen{w^n}\) acts transitively on \(\sgen{v, w^n}\bslash\omega\); in particular, that \card{\sgen{v, w^n}\bslash\omega} divides \(\card{\sgen w/\sgen{w^n}} = n\), hence is odd.  Thus,
\begin{multline}
\tag{$*$}
\sublabel{eq:decompose}
\card{\sgen{v, w^n}\bslash\Root_{\theta^n}}
= \sum_{\omega \in \sgen{v, w}\bslash\Root_\theta}
	\card{\sgen{v, w^n}\bslash\omega} \\
\equiv \sum_{\omega \in \sgen{v, w}\bslash\Root_\theta} 1
= \card{\sgen{v, w}\bslash\Root_\theta}
\pmod{2\Z}.
\end{multline}

Now use \subeqref{eq:decompose} twice (once with \(v = 1\) and once with \(v = \ol\sigma\)) and Lemma \ref{lem:perm-sign} to conclude that
\[
\Legend\sigma\theta = (-1)^{\card{\sgen w\bslash\Root_\theta} - \card{\sgen{v, w}\bslash\Root_\theta}} \\
= (-1)^{\card{\sgen{w^n}\bslash\Root_{\theta^n}} - \card{\sgen{v, w^n}\bslash\Root_{\theta^n}}} = \Legend\sigma{\theta^n}.\qedhere
\]
\end{proof}

For the reduction result Proposition \ref{prop:reduce-elliptic}, we need the notion of the rank of a group relative to a twisted automorphism.

\begin{defn}
If \(T\) is a maximal torus and \(\xi\) is a twisted automorphism of \(T\), then we define \(\rk_\xi T\) to be the dimension of the space of fixed points of \(\xi\) in the character lattice of \(T\).  If \(H\) is a connected, reductive \FF-group and \(\xi\) is a twisted automorphism of \(H\), then we define \mnotn{\rk_\xi H} to be the maximum value of \(\rk_\xi T\) over all \(\xi\)-stable maximal tori \(T\) in \(H\).
\end{defn}

\begin{prop}[\xcite{spice:signs-alg}*{Corollary \xref{cor:qs-rank-by-orbits} and Proposition \xref{prop:linear-rank-by-unip}}]
\label{prop:reduce-elliptic}
If \(M\) is a \(\sigma\)- and \(\theta\)-stable Levi component of a \(\theta\)-stable parabolic subgroup of \(G\), then
\[
\Legend\sigma\theta_{\!\Root(G, T)} = (-1)^{\rk_\sigma (G^\theta)\conn - \rk_\sigma (M^\theta)\conn}\Legend\sigma\theta_{\!\Root(M, T)}.
\]
\end{prop}

\begin{rem}
\label{rem:reduce-elliptic}
With the notation of Corollary \ref{cor:Borel} and \springer*{\S15.4, p.~264}, we may take \(M = L_\subSimple\) in Proposition \ref{prop:reduce-elliptic}.  Notice that, under the conditions of the proposition, the groups \(G^\theta\) and \(M^\theta\) are \(\sigma\)-stable, even though \(\sigma\) and \(\theta\) may not commute.
\end{rem}

\begin{rem}
\label{rem:elliptic-list}
Finally, we are in a position to reduce to a small number of cases.  By Remark \ref{rem:reduce-irreducible}, we may assume that \Root is irreducible.  Let \(\theta\) be any automorphism.  By Proposition \ref{prop:reduce-elliptic} and Remark \ref{rem:reduce-elliptic}, we may replace \Root by a Levi subsystem on which \ol\theta acts as an elliptic automorphism.  By Proposition \ref{prop:A+-} and Corollary \ref{cor:elliptic-cohom}, upon passing to the adjoint quotient of \(G\) if necessary, we may assume that \(\theta\) is semi-pinned; so, by Lemma \ref{lem:reduce-2-power}, we may replace \(\theta\) by an odd power, and so assume that \ol\theta has \(2\)-power order.  However, \ol\theta may now no longer be elliptic, so we may use Proposition \ref{prop:reduce-elliptic} again to reduce to the case where \ol\theta is both elliptic and of \(2\)-power order.

In the notation of \carter*{p.~41}, an element \(w \in W\) is elliptic if and only if \(\ol W_1 = W\).

By \carter*{Table 3}, an element \(w \in W\) has \(2\)-power order if and only if every admissible diagram \carter*{\S4, p.~7} attached to it, in the sense of \carter*{p.~6}, has only components of the forms \An[2^r - 1], \(\Bn[2^r] = \Cn[2^r]\), \Dn[2^r + 1], \(\Dn[2^r + 2^s](a_{2^s - 1})\), and \(\Dn[2^{s + 1}](b_{2^s - 1})\) (which is the same as \(\Dn[2^{s + 1}](a_{2^s - 1})\)) with \(r \ge 0\) and \(s > 0\).

We handle the possibilities for the classical groups nearly uniformly in \S\ref{sec:classical}.  By \carter*{Tables 7--11}, the elliptic conjugacy classes of \(2\)-power order in the exceptional Weyl groups are:
	\begin{itemize}
	\permunwrap
	\item \(\An[1] \times \tilde{\type A}_1 = -1\) in type \Gn;
	\item \(\An[1]^4 = -1\), \(\An[3] \times \tilde{\type A}_1\), \(\Dn[4](a_1)\), and \Bn[4] in type \Fn;
	\item \(\An[1]^7 = -1\), \(\An[3]^2 \times \An[1]\), and \An[7] in type \En[7];
	\item \(\An[1]^8 = -1\), \(\An[3]^2 \times \An[1]^2\), \(\An[7] \times \An[1]\), \(\Dn[4](a_1)^2\), and \(\Dn[8](a_3)\) in type \En[8].
	\rewrap
	\end{itemize}
There are none in \(W(\En[6])\).  Since \(\Aut(\En[6]) = W(\En[6]) \times \sgen{-1}\), a conjugacy class in \(\Aut(\En[6])\) is elliptic if and only if it consists of elements of the form \(-w\), where \(w \in W(\En[6])\) has no eigenvalue of order \(-1\).  By \carter*{Table 3} again, the only way that this can happen for \(w\) of \(2\)-power order is if every admissible diagram attached to \(w\) has only components of the forms \(\Bn[2^{r + 1}] = \Cn[2^{r + 1}]\), \(\Dn[2^r + 2^s](a_{2^s - 1})\), and \(\Dn[2^{s + 1}](b_{2^s - 1}) = \Dn[2^{s + 1}](a_{2^s - 1})\) with \(r \ge 0\) and \(s > 0\).  Thus, by \carter*{Table 9} again, the only elliptic conjugacy classes in \(\Aut(\En[6])\) of \(2\)-power order are \(-1\) and \(-\Dn[4](a_1)\).
\end{rem}

Our calculations (see particularly Remark \ref{rem:Rw-full}) show that, whenever \Root is irreducible and \ol\theta is elliptic and of \(2\)-power order, we have that \(\Root_\theta = \Root\), \emph{except} if \Root is of type \Cn, or \Root is of type \Fn and \(w\) is of type \(\An[3] \times \tilde{\type A}_1\) (see \S\ref{sec:A3xA1-in-F4}.  We summarise the calculations for types \An, \Fn, and \En in Tables \ref{tab:An}--\ref{tab:E6}.  The results for \Bn, \Cn, and \Dn are slightly more complicated to state; see \S\ref{sec:Bn-Cn-Dn}.  For \Gn, the one possibility has already been discussed; see Remark \ref{rem:dual-Coxeter}.

\begin{table}
\caption{$\mathsf A_{n - 1}$}
\label{tab:An}
\[\begin{array}{c|c|c}
\text{Conjugacy class of \(w\)} & \ker \Legend\cdot w & \text{Details} \\ \hline
\text{Coxeter, \(n \equiv 0, 3 \pmod4\)} & C_A(w) \rtimes \smashed\set{\norm_{\An, q}}{\sgn^+_n(q) = 1} & \text{\S\ref{sec:An}} \\
\text{Coxeter, \(n \equiv 1, 2 \pmod4\)} & C_W(w) \rtimes \smashed\set{\pm\norm_{\An, q}}{\sgn^+_n(q) = \pm1} & \text{\S\ref{sec:An}} \\
\text{\(-1\), \(n\) even} & A & \text{Remark \ref{rem:dual-Coxeter}} \\
\text{\(-1\), \(n\) odd} & \ker \sgn_{\An} \times \sgen{-1} & \text{Remark \ref{rem:dual-Coxeter}}
\end{array}\]
\end{table}

\begin{table}
\label{tab:F4}
\caption{$\mathsf F_4$}
\[\begin{array}{c|c|c}
\text{Conjugacy class of \(w\)} & \ker \Legend\cdot w & \text{Details} \\ \hline
-1 & \ker \sgn_{\Fn} & \text{Remark \ref{rem:dual-Coxeter}} \\
\An[3] \times \tilde{\type A}_1 & \bigl(\sgen{\ncycle(0)} \times \sgen{\ncycle(1)} \times \sgen{\ncycle(2 3)}\bigr) \rtimes \smashed\sgen{\cycle(0 1)\norm_{\Bn[2 + 1^{+2}], -1}} & \text{\S\ref{sec:A3xA1-in-F4}} \\
\Dn[4](a_1) & N_A(\sgen w) & \text{\S\ref{sec:D4(a1)-in-F4}} \\
\Bn[4] & C_A(w) \rtimes \smashed\sgen{\norm_{\Bn[4], -1}} & \text{\S\ref{sec:B4-in-F4}}
\end{array}\]
\end{table}

\begin{table}
\label{tab:E8}
\caption{$\mathsf E_8$}
\[\begin{array}{c|c|c}
\text{Conjugacy class of \(w\)} & \ker \Legend\cdot w & \text{Details} \\ \hline
-1 & A & \text{Remark \ref{rem:dual-Coxeter}} \\
\unwrap\An[3]^2 \times \unwrap\An[1]^2 & C_A(w) & \text{\S\ref{sec:A3xA3xA1xA1-in-E8}} \\
\An[7] \times \An[1] & C_A(w) \rtimes \smashed\sgen{\norm_{\Dn[4 + 2 + 1^{+2}], 3}} & \text{\S\ref{sec:A7xA1-in-E8}} \\
\Dn[4](a_1)^2 & N_A(\sgen w) & \text{\S\ref{sec:D4(a1)xD4(a1)-in-E8}} \\
\Dn[8](a_3) & C_A(w) \rtimes \smashed\sgen{\norm_{\Dn[4^{+2}], -3}} & \text{\S\ref{sec:D8(a3)-in-E8}}
\end{array}\]
\end{table}

\begin{table}
\label{tab:E7}
\caption{$\mathsf E_7$}
\[\begin{array}{c|c|c}
\text{Conjugacy class of \(w\)} & \ker \Legend\cdot w & \text{Details} \\ \hline
-1 & A & \text{Remark \ref{rem:dual-Coxeter}} \\
\unwrap\An[3]^2 \times \An[1] & C_A(w) & \text{\S\ref{sec:A3xA3xA1xA1-in-E8}} \\
\An[7] & C_A(w) \rtimes \smashed\sgen{\norm_{\Dn[4 + 2], 3}} & \text{\S\ref{sec:A7xA1-in-E8}}
\end{array}\]
\end{table}

\begin{table}
\caption{$\mathsf E_6$}
\label{tab:E6}
\[\begin{array}{c|c|c}
\text{Conjugacy class of \(w\)} & \ker \Legend\cdot w & \text{Details} \\ \hline
-1 & A & \text{Remark \ref{rem:dual-Coxeter}} \\
-D_4(a_1) & C_A(w) & \text{\S\ref{sec:A3xA3xA1xA1-in-E8}}
\end{array}\]
\end{table}

\numberwithin{equation}{section}

\section{Classical groups}
\label{sec:classical}

As in \S\ref{sec:signs}, we let \FF be a separably closed field.

We follow essentially the notation of \bourbaki*{Planches VI.I--IV}, except that we find it useful to use \(0\)- (rather than \(1\)-) based indexing.  Thus, we realise \Bn, \Cn, and \Dn as root systems in an \(n\)-dimensional Euclidean space \(\mathbb E^n\) with orthonormal basis \(\set{\mnotn{\basis_i}}{i \in \Z/n\Z}\).  We define \(\tbasis_i\) for \(i \in \Z/2n\Z\) by putting
\[
\mnotn{\tbasis_i} = \begin{cases}
\basis_i,        & 0 \le i < n   \\
-\basis_{i - n}, & n \le i < 2n,
\end{cases}
\]
and then reading subscripts of \(\tbasis\) modulo \(2n\).  Then \Cn and \Dn consist of the vectors \(\tbasis_i - \tbasis_j\) with \(i, j \in \Z/2n\Z\) such that (for \Cn) \(i \ne j\) and (for \Dn) the images of \(i\) and \(j\) in \(\Z/n\Z\) are distinct; and \(\Bn = \Dn \cup \set{\tbasis_i}{i \in \Z/2n\Z}\).  We also realise \An in a subspace of \(\mathbb E^n\) as the set \(\sett{\basis_i - \basis_j}{\(i, j \in \Z/n\Z\) and \(i \ne j\)}\).  We regard elements of \(W(\An)\) as orthogonal transformations of \(\mathbb E^n\) by letting them act trivially on the (\(1\)-dimensional) orthogonal complement of \An.

We identify \(W(\An)\) with \(\Sgp_n\) (so that \(\sgn_n = \sgn_{\An}\)), \(W(\Bn) = W(\Cn)\) with \(\sgen{-1}^{\oplus\Z/n\Z} \rtimes \Sgp_n\), and \(W(\Dn)\) with \(\ker \sgn\textsub{\Bn, short} = \ker \sgn\textsub{\Cn, long}\).  Explicitly,
\[
W(\Dn) = \set{\bigoplus_{i \in \Z/n\Z} \epsilon_i \rtimes \sigma}{\prod_{i \in \Z/n\Z} \epsilon_i = 1}.
\]
It is also helpful to identify \(W(\Bn)\) with the group of permutations of \(\Z/2n\Z\) that commute with \anonmapto i{i + n}.  Namely, \((\varepsilon_i)_{i \in \Z/n\Z} \rtimes w\) is identified with the permutation that, for \(i = 0, \dotsc, n - 1\), sends \(i\) to \(w i\) and \(i + n\) to \(w i + n\) if \(\varepsilon_i = 1\), and that sends \(i\) to \(w i + n\) and \(i + n\) to \(w i\) if \(\varepsilon_i = -1\).  Suppose that \(j_1, \dotsc, j_r \in \Z/2n\Z\) are \(r\) distinct elements such that \(\sset{j_1, \dotsc, j_r} \cap \sset{j_1 + n, \dotsc, j_r + n} = \emptyset\).  Then, following \carter*{\S7, p.~25}, we abbreviate \(\cycle({{j_1} \dotsc {j_r}} {{j_1 + n} \dotsc {j_r + n}})\) to \(\ncycle({{j_1} \dotsc {j_r}})\), and call it a \term{negative cycle}; and, by abuse of notation and terminology, abbreviate \(\cycle({{j_1} \dotsc {j_r}})\cycle({{j_1 + n} \dotsc {j_r + n}})\) to just \(\cycle({{j_1} \dotsc {j_r}})\), and call it a \term{positive cycle}.

Note that \An[1], \Bn[1], and \Cn[1] are isomorphic; \Dn[1] is empty; and \Dn[2] is isomorphic to \(\An[1] \sqcup \An[1]\).

Now we construct some special elements of the classical Weyl groups, and lifts to the classical Lie groups, that will be useful in our calculations below.

\begin{defn}
Let \mnotn{\norm_{\An, q}} (for \(q \in (\Z/n\Z)\mult\)) be the unique element of \(W(\An)\) such that \(\norm_{\An, q}\basis_i = \basis_{q i}\) for all \(i \in \Z/n\Z\); and let \mnotn{\norm_{\Cn, q}} (for \(q \in (\Z/2n\Z)\mult\)) be the unique element of \(\Aut(\Cn)\) such that \(\norm_{\Cn, q}\tbasis_i = \tbasis_{q i}\) for all \(i \in \Z\).

If \(\lambda = m_0 + \dotsb + m_t\) is a partition of \(n\), then \(\mnotn{\Cn[\lambda]} \ldef \Cn[m_0] \coprod \dotsb \coprod \Cn[m_t]\) sits naturally inside \Cn, and furnishes an embedding of \(W(\Cn[\lambda]) = \bigoplus_{i = 0}^t W(\Cn[m_i])\) in \(W(\Cn)\).  If \(q\) is odd and relatively prime to all parts of \(\lambda\), then we write \mnotn{\norm_{\Cn[\lambda], q}} for the image of \(\bigoplus_{i = 0}^t \norm_{\Cn[m_i], q}\) in \(W(\Cn)\).  (Strictly speaking, these notations are defined only up to \(W(\Cn)\)-conjugacy, but the ambiguity should cause no confusion.)  When convenient, we will also write \mnotn{\norm_{\Bn[\lambda], q}} when we view this transformation as an element of \(W(\Bn)\), or \mnotn{\norm_{\Dn[\lambda], q}} when we view it as an element of \(\Aut(\Dn)\).
\end{defn}

\begin{rem}
\label{rem:cross}
Every element of \Bn and \Cn (hence of \An and \Dn) lies in \(\Q\Root' \oplus \Q\Root''\) for some irreducible components \(\Root'\) and \(\Root''\) of \Cn[\lambda].
\end{rem}

Now let \(V = \mnotn{V_{\An}}\) be an \(n\)-dimensional vector space, say with (ordered) basis \((v_i)_{i \in \Z/n\Z}\).  We write \(T_{\An}\) for the maximal torus in \(\GL(V_{\An})\) that acts on each \(v_i\) by scalar multiplication.

\begin{defn}
Let \mnotn{\dot w_{\An}} be the element of \(\GL(V_{\An})(\FF)\) that sends \(v_i\) to \(v_{i + 1}\) for \(i \in \Z/n\Z\).  Then \(\dot w_{\An}\) normalises \(T_{\An}\), and projects to a Coxeter element \(w_{\An}\) of \(W(\GL(V_{\An}), T_{\An}) \cong W(\An)\) for which \(\norm_{\An, q}w_{\An} = w_{\An}^q\norm_{\An, q}\).
\end{defn}

Write
\mnotn{V_{\Bn}}, \mnotn{V_{\Cn}}, and \mnotn{V_{\Dn}}, respectively,
for the spaces
\(V \oplus \mf{gl}_1 \oplus \dual V\), \(V \oplus \dual V\), and \(V \oplus \dual V\), respectively,
given the natural symmetric, alternating, and symmetric pairing, respectively, \anonpair
such that \(\pair{\dual v}v = \dual v(v)\) for all \(v \in V\) and \(\dual v \in \dual V\) (and, for type \Bn, so that \(\gl_1\) is orthogonal to \(V \oplus \dual V\), and carries the multiplication pairing).  If we need to emphasise the type of the pairing, then we shall write \(\anonpair_\Root\) instead of just \anonpair, where \Root is \Bn, \Cn, or \Dn.

Let \((\unwrap\dual v_i)_{i \in \Z/n\Z}\) be the basis of \dual V such that
\[
\unwrap\dual v_i(v_j) = \begin{cases}
2, & i = j   \\
0, & i \ne j
\end{cases}\qforall{\(i, j \in \Z/n\Z\),}
\]
and write \(e\) for \(1\), viewed as an element of \(\gl_1 \subseteq V_{\Bn}\).  Write \(T_{\Dn} = T_{\Bn}\) for the maximal torus in \(\Spin(V_{\Dn}) \subseteq \Spin(V_{\Bn})\) consisting of those elements that (\via the natural action on \(V_{\Dn} \subseteq V_{\Bn}\)) act on each \(v_i\) and \(\unwrap\dual v_i\) by scalar multiplication.

\begin{defn}
Write
\begin{align*}
\mnotn{\dot w_{\Dn}} ={} & \frac{v_0 - \unwrap\dual v_0}{2\sqrt{-1}}\dotm\frac{v_0 - v_1 + \unwrap\dual v_0 - \unwrap\dual v_1}{2\sqrt2}\dotm\frac{v_0 - v_1 - \unwrap\dual v_0 + \unwrap\dual v_1}{2\sqrt{-2}}\times\dotsb\times{} \\
               & \qquad\frac{v_{n - 2} - v_{n - 1} + \unwrap\dual v_{n - 2} - \unwrap\dual v_{n - 1}}{2\sqrt2}\dotm\frac{v_{n - 2} - v_{n - 1} - \unwrap\dual v_{n - 2} + \unwrap\dual v_{n - 1}}{2\sqrt{-2}} \\
           \in{} & \Pin(V_{\Dn})(\FF) \subseteq \Pin(V_{\Bn})(\FF) \\
\intertext{and}
\mnotn{\dot w_{\Bn}} ={} & \dot w_{\Dn}\dotm e \in \Spin(V_{\Bn})(\FF).
\end{align*}
(Here we are thinking of the pin group of a quadratic space \mc V as generated by products in the Clifford algebra of unit vectors \mc V.  See, for example, \cite{atiyah-bott-shapiro:clifford}*{\S3}, particularly \loccit*{Theorem 3.11}.)  The element \(\dot w_{\Dn}\) depends on the choice of \(\sqrt{-1}\) if \(n\) is odd, but not on that of \(\sqrt2\) (since it appears an even number of times).
\end{defn}

The elements \(\dot w_{\Dn}\) and \(\dot w_{\Bn}\) normalise \(T_{\Dn} = T_{\Bn}\), and their common image in \(W(\Bn)\) is a Coxeter element there (\emph{not} in \(W(\Dn)\)!).

\begin{defn}
If \(n = m' + m''\), then, \textit{via} the natural embedding of \(\Pin(V_{\Dn[m']}) \times \Pin(V_{\Dn[m'']})\) in \(\Pin(V_{\Dn})\), the element \(\dot w_{\Dn[m']} \times \dot w_{\Dn[m'']}\) is carried into \(\Spin(V_{\Dn})(\FF)\).  We denote its image there (which no longer depends on a choice of \(\sqrt{-1}\)!\@) by \(\dot w_{\Dn[m' + m'']}\).  If \(\lambda = \lambda' + m' + m''\) is a partition of \(n\), then we embed \(\Spin(V_{\Bn[\lambda']}) \times \Spin(V_{\Dn[m' + m'']})\) naturally in \(\Spin(V_{\Bn})\), and \(\Pin(V_{\Dn[\lambda']}) \times \Spin(V_{\Dn[m' + m'']})\) in \(\Pin(V_{\Dn})\), and write \mnotn{\dot w_{\Bn[\lambda]}} and \mnotn{\dot w_{\Dn[\lambda]}} for the corresponding images of \(\dot w_{\Bn[\lambda']} \times \dot w_{\Dn[m' + m'']}\) and \(\dot w_{\Dn[\lambda']} \times \dot w_{\Dn[m' + m'']}\), respectively.
\end{defn}

\begin{rem}
\label{rem:char-values}
When viewed as transformations of \(V_{\Dn}(\FF)\) and \(V_{\Bn}(\FF)\), respectively, \(\dot w_{\Dn}\) and \(\dot w_{\Bn}\) both send \(v_i\) to \(v_{i + 1}\) and \(\unwrap\dual v_i\) to \(\unwrap\dual v_{i + 1}\) for \(i \ne n - 1\), and send \(v_{n - 1}\) to \(\unwrap\dual v_0\) and \(\unwrap\dual v_{n - 1}\) to \(v_0\).  To determine completely the action of \(\dot w_{\Bn}\) on \(V_{\Bn}\), we need only observe that it negates \(e\).

In particular, if \(\zeta^{2n} = 1\), then the \(\zeta\)-eigenspaces of \(\dot w_{\Dn}\) in \(V_{\Dn}(\FF)\) and of \(\dot w_{\Bn}\) in \(V_{\Bn}(\FF)\) both contain \(e_\zeta \ldef \sum_{i = 0}^{n - 1} (\zeta^{-i}v_i + \zeta^{-n}\unwrap\dual v_i)\).  They are in fact equal to the span of this vector, except that the \(-1\)-eigenspace of \(\dot w_{\Bn}\) is spanned by \(e_{-1}\) and \(e\).

In particular, if \(S_{\Bn}\) is the torus in \(\SO(V_{\Bn})\) (\emph{not} the spin group) consisting of those elements that act on each \(e_\zeta\) by scalar multiplication, and \(\zeta_{2n}\) is a primitive \(2n\)th root of unity in \FF, then \(\dot w_{\Bn} \in S_{\Bn}(\FF)\), and there is a basis \set{\basis_0, \dotsc, \basis_{n - 1}} of the character lattice of \(S_{\Bn}\) such that \(\basis_i(\dot w_{\Dn}) = \basis_i(\dot w_{\Bn}) = \zeta_{2n}^i\) for \(i \in \sset{0, \dotsc, n - 1}\).  (Note that \(\dot w_{\Bn}\) is just the extension of \(\dot w_{\Dn}\) to the map on \(V_{\Bn} = V_{\Dn} \oplus \gl_1\) that acts by negation on \(\gl_1\).)
\end{rem}

\begin{defn}
Let \mnotn{\dot w_{\Cn}} be the element of \(\Sp(V_{\Cn})(\FF)\) that sends \(v_i\) to \(v_{i + 1}\) and \(\unwrap\dual v_i\) to \(\unwrap\dual v_{i + 1}\) for \(i \ne n - 1\), and sends \(v_{n - 1}\) to \(\unwrap\dual v_0\) and \(\unwrap\dual v_{n - 1}\) to \(-v_0\).
\end{defn}

Then \(\dot w_{\Cn}\) normalises the maximal torus \(T_{\Cn}\) in \(\Sp(V_{\Cn})\) consisting of those elements that act on each \(v_i\) and \(\unwrap\dual v_i\) by scalar multiplication, and it projects to a Coxeter element of \(W(\Sp(V_{\Cn}), T_{\Cn}) \cong W(\Cn)\).

\begin{defn}
If \(\lambda = m_0 + \dotsb + m_t\) is a partition of \(n\), then we write \mnotn{\dot w_{\Cn[\lambda]}} for the image of \(\prod_{i = 0}^t \dot w_{\Cn[m_i]}\) under the natural embedding of \(\prod_{i = 0}^t \Sp(V_{\Cn[m_i]})\) in \(\Sp(V_{\Cn})\).
\end{defn}

The images of \(\dot w_{\Bn[\lambda]}\), \(\dot w_{\Cn[\lambda]}\), and \(\dot w_{\Dn[\lambda]}\) in \(W(\Bn) = W(\Cn) \subseteq \Aut(\Dn)\) are all the same element \(\mnotn{w_{\Bn[\lambda]}} = \mnotn{w_{\Cn[\lambda]}} = \mnotn{w_{\Dn[\lambda]}}\), and satisfy \(\norm_{\Cn[\lambda], q}w_{\Cn[\lambda]} = w_{\Cn[\lambda]}^q\norm_{\Cn[\lambda], q}\).

\begin{rem}
\label{rem:Rw-full}
Let \((G, T)\) be \((\GL(V_{\An}), T_{\An})\), \((\SO(V_{\Bn}), T_{\Bn}/\sset{\pm1})\), \((\Sp(V_{\Cn}), T_{\Cn})\), or \((\SO(V_{\Dn}), T_{\Dn}/\sset{\pm1})\).  Then \(N_G(T)\) consists of generalised permutation matrices with respect to the basis given above for the underlying space \(V_\Root\).  The following claims are easily verified.

The stabiliser of \(\basis_i - \basis_j\) in \(N_G(T)\) consists of those generalised permutation matrices whose \((v_i, v_i)\) and \((v_j, v_j)\) entries are non-\(0\) and equal.

The stabiliser of \(\basis_i + \basis_j\) in \(N_G(T)\) (for \Root of type \Bn, \Cn, or \Dn) consists of those generalised permutation matrices for which the product of the \((v_i, v_i)\) and \((v_j, v_j)\) entries is \(1\).

The stabiliser of \(\basis_i\) in \(N_G(T)\) (for \Root of type \Bn) consists of those generalised permutation matrices for which the \((v_i, v_i)\) and \((e, e)\) entries are equal.  (Note that the \((e, e)\) entry must be \(+1\) or \(-1\).)

The stabiliser of \(2\basis_i\) in \(N_G(T)\) (for \Root of type \Cn) consists of those generalised permutation matrices for which the square of the \((v_i, v_i)\) entry is \(1\).

It follows from the above that \((\An)_{\dot w_{\An}} = \An\) and \((\Dn)_{\dot w_{\Dn[\lambda]}} = \Dn\).  Once we note that \(\dot w_{\Bn[\lambda]}^e\) can fix a root of the form \(\basis_i\) only when \(2 \mid e\), it also follows that \((\Bn)_{\dot w_{\Bn[\lambda]}} = \Bn\).  The situation for type \Cn is a bit more complicated.  Namely, \((\Cn)_{\dot w_{\Cn[\lambda]}} = \Cn[\ol\lambda]\), where \(\ol\lambda\) is the partition of \(n\) that arises by lumping together all parts of \(\lambda\) with the same \(2\)-adic valuation.  To be more precise, if \(\lambda = m_0 + \dotsb + m_t\), then we put \(\ol\lambda = \ol m_0 + \ol m_1 + \dotsb\), where \(\ol m_j\) is the sum of all parts \(m_i\) of \(\lambda\) with \(2\)-adic valuation \(j\).  (For example, \(\ol m_0\) is the sum of all odd parts of \(\lambda\).)
\end{rem}

\begin{rem}
\label{rem:Dn-as}
We shall frequently use the fact that, if \root and \otherroot are orthogonal vectors in a Euclidean space \(V\), then \(\refl_\root\refl_\otherroot = \refl_{\root + \otherroot}\refl_{\root - \otherroot}\).  In particular, consider the admissible diagram \(\Dn(a_s)\) in \Dn.  With the notation of \bourbaki*{Planche VI.IV}, put \(\root' = \root_{s + 1} + 2\root_{s + 2} + \dotsb + 2\root_{n - 2} + \root_{n - 1} + \root_n\).  We wish to find the conjugacy class \(\mc C_{\Dn}\) in \(W(\Dn)\) attached to the admissible diagram
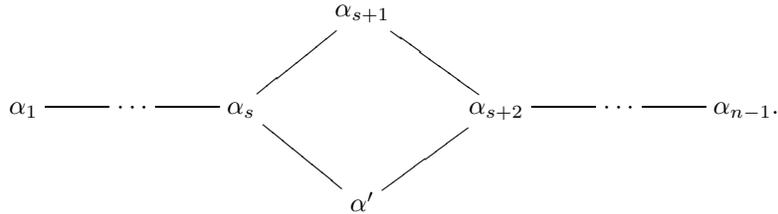
\[\xymatrix{
& & & \root_{s + 1} \ar@{-}[dr] \\
\root_1 \ar@{-}[r] & \cdots \ar@{-}[r] & \root_s \ar@{-}[ur]\ar@{-}[dr] & & \root_{s + 2} \ar@{-}[r] & \cdots \ar@{-}[r] & \root_{n - 1}. \\
& & & \root' \ar@{-}[ur] \\
}\]
Put \(\otherroot = \frac1 2(\root_{s + 1} + \root')\) and \(\otherroot' = \frac1 2(\root_{s + 1} - \root')\), so that \(\otherroot, \otherroot' \in \Bn\).  Then \(\refl_{\root_{s + 1}}\refl_{\root'} = \refl_\otherroot\refl_{\otherroot'}\), so \(\mc C_{\Dn}\) is contained in the conjugacy class \(\mc C_{\Bn}\) in \(W(\Bn)\) attached to the admissible diagram
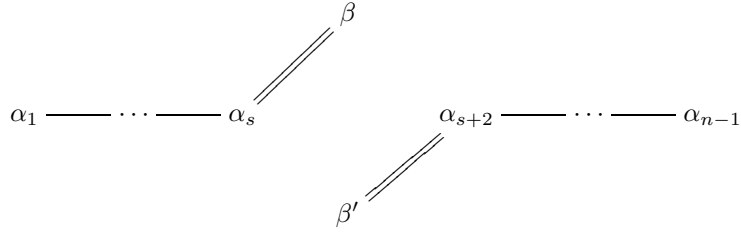
\[\xymatrix{
& & & \otherroot \\
\root_1 \ar@{-}[r] & \cdots \ar@{-}[r] & \root_s \ar@{=}[ur] & & \root_{s + 2} \ar@{-}[r] & \cdots \ar@{-}[r] & \root_{n - 1} \\
& & & \otherroot' \ar@{=}[ur]
}\]
of type \(\Bn[s + 1] \times \Bn[n - s - 1]\).  This is attached to the product of the Coxeter classes in \Bn[s + 1] and \Bn[n - s - 1], \ie, the conjugacy class of \(w_{\Bn[\lambda]}\), where \(\lambda\) is the partition \((s + 1) + (n - s - 1)\) of \(n\).  By \carter*{Proposition 25}, we have that \(\mc C_{\Bn}\) is a \(W(\Dn)\)-conjugacy class, so \(\mc C_{\Dn} = \mc C_{\Bn}\).
\end{rem}

\begin{rem}
\label{rem:unbalanced}
Fix \(w \in W(\Cn)\).  We will temporarily say that an element of \(\R\Cn\) is \(w\)-unbalanced if the sum of the elements of its \(w\)-orbit is non-\(0\).  If \(w\) preserves a root system \(\Root \subseteq \R\Cn\), then the image of \(w\) in \(\Aut(\Root)\) is elliptic if and only if there are no \(w\)-unbalanced roots in \Root.

If \(i\) belongs to a positive cycle of \(w\), then \(\tbasis_i \in \Bn\) and \(2\tbasis_i \in \Cn\), and (if \(n > 1\) and \(j \ne i, i + n\) is arbitrary) \(\tbasis_i + \tbasis_j \in \Dn\), are \(w\)-unbalanced.

If \(w \in W(\An)\) and \(i\) and \(j\) belong to different cycles of \(w\) (on \(\Z/n\Z\)), then \(\basis_i - \basis_j\) is \(w\)-unbalanced.  If \(w = -\sigma \in -W(\An)\), and \(i\) belongs to an even-length cycle of \(\sigma\), then \(\basis_i - \basis_{\sigma i}\) is \(w\)-unbalanced.
\end{rem}

\subsection{Type \texorpdfstring{$\mathsf A_{n - 1}$}{An - 1}}
\label{sec:An}

The results of this section are summarised in Table \ref{tab:An}.

Put \(\Root = \An\), \(W = W(\Root)\), and \(A = \Aut(\Root)\).  Let \(w\) be an elliptic element of \(A\).

By regarding \(w\) as an element of \(W(\Bn)\), we see from Remark \ref{rem:unbalanced} that it is a Coxeter element in \(W\), or else of the form \(w = -\sigma\), where \(\sigma \in W\) has odd order.  In the latter case, if \(w\) has \(2\)-power order, then \(w = -1\), and we may use Proposition \ref{prop:minus-sign}.

Suppose that \(w = w_\Root\).  Then
\[
C_A(w) = \sgen w \times \sgen{-1}
\qandq
N_A(\sgen w) = C_A(w) \rtimes \set{\norm_{\Root, q}}{q \in (\Z/n\Z)\mult}.
\]
There are \(n - 1\) orbits of \(w\) on \Root, each of size \(n\), and each containing exactly one vector of the form \(\basis_0 - \basis_j\).  The element \(\norm_{\Root, q}\) sends \(\basis_0 - \basis_j\) to \(\basis_0 - \basis_{q j}\), so, by Definition \ref{defn:sgn+-},
\[
\Legend{\norm_{\Root, q}}w = \sgn^+_n(q).
\]
Since \(-\norm_{\Root, -1}\) stabilises every \(w\)-orbit,
\[
\Legend{-1}w = \Legend{\norm_{\Root, -1}}w = \sgn^+_n(-1).
\]

\subsection{Types \texorpdfstring{$\mathsf B_n$, $\mathsf C_n$, and $\mathsf D_n$}{Bn, Cn, and Dn}}
\label{sec:Bn-Cn-Dn}

Suppose that \Root is one of \Bn or \Cn (the `non-simply laced cases') or \Dn (the `simply laced case').  Put \(W = W(\Root)\) and \(A = W(\Bn)\).  Then \(A = \Aut(\Root)\) unless \(\Root = \Dn[4]\).  In this case, every conjugacy class in \(\Aut(\Root)\) either intersects \(A\), or has order divisible by \(3\); and \(A\) has index \(3\) in \(\Aut(\Root)\), which means that, for any \(w \in A\), the subgroup \(N_A(\sgen w)\) of \(N_{\Aut(\Root)}(\sgen w)\) has index dividing \(3\), hence that \Legend\cdot w is determined by its restriction to \(N_A(\sgen w)\).

Let \(w\) be an elliptic element of \(A\).  By Remark \ref{rem:unbalanced}, every cycle of \(w\) is negative.  That is, there is a partition \(\lambda\) of \(n\) such that \(w = w_{\Cn[\lambda]}\); so \(C_A(w) = \bigoplus_{i = 1}^\infty \sgen{w_{\Cn[i]}}^{\oplus \lambda_i} \rtimes \bigoplus_{i = 1}^\infty \Sgp_{\lambda_i}\), and
\[
N_A(\sgen w) = C_A(w) \rtimes \set{\norm_{\Cn[\lambda], q}}{q \in (\Z/e\Z)\mult},
\]
where \(e\) is the order of \(w\).

By Remark \ref{rem:Rw-full}, if \Root is of type \Bn or \Dn, then \(\Root_w = \Root\); whereas, if \Root is of type \Cn, then \(\Root_w = \Cn[\ol\lambda]\), where \(\ol\lambda = \ol m_0 + \ol m_1 + \dotsb\) is as in that remark.  In the latter case, we have that \(N_A(\sgen w) = \prod_{j = 0}^\infty N_{W(\Cn[\ol m_j])}(\sgen w)\) and
\[
\Legend{\bigoplus v_j}w_{\!\Root} = \prod_{j = 0}^\infty \Legend{v_j}w_{\!\Cn[\ol m_j]}.
\]
Thus there is no loss of generality in replacing \Cn by an irreducible component of \Cn[\ol\lambda], or, in other words, of assuming that all parts of \(\lambda\) have the same \(2\)-adic valuation.  Upon doing so, we once again have that \(\Root_w = \Root\).

Let \(\Root'\) be an irreducible component of \Cn[\lambda], of type \Cn[r'], say, and choose labels so that \(\Root' = \sett{\tbasis'_i - \tbasis'_j}{\(i, j \in \Z/2r'\Z\) and \(i \ne j\)}\).  Recall that there is associated to \(\Root'\) a Coxeter element \(w_{\Root'} = w_{\Cn[r']}\).  Then \(w\) has \(r'\) orbits on \(\Root \cap \Q\Root'\) in the non-simply laced cases, and \(r' - 1\) orbits on \(\Root \cap \Q\Root'\) in the simply laced case, each of size \(2r'\); they are parameterised by sets of up to two mutually inverse, non-\(0\) (and non-\(r'\), in the simply laced case) elements of \(\Z/2r'\Z\), and the roots of the form \(\tbasis'_0 - \tbasis'_i\) lying in the orbit corresponding to \sset{j, -j} are precisely those for which \(i \in \sset{j, -j}\).

The element \(w_{\Root'}\) fixes pointwise the orthogonal complement of \(\Root'\), and the orbits of \(w\) on \(\Q\Root'\).  The element \(\norm_{\Cn[\lambda], q}\) sends \(\tbasis'_0 - \tbasis'_j\) to \(\tbasis'_0 - \tbasis'_{q j}\), hence, by Proposition \ref{prop:Legendre}, acts on the space of orbits in \(\Root'\) as a permutation of sign \Legend{2r'}q if \(r'\) is even, and of sign \(1\) if \(r'\) is odd.  (This is true even in the simply laced case, since multiplication by \(q\) fixes \(r'\).)

Let \(\Root''\) be a different irreducible component of \Cn[\lambda], of type \Cn[r''], say, and choose labels so that a typical vector of \(\Root''\) is of the form \(\tbasis''_i - \tbasis''_j\).  (We may have \(r' = r''\).)

The elements of \(\Root \cap (\Q\Root' \oplus \Q\Root'')\) that do not lie in \(\Q\Root'\) or \(\Q\Root''\) are precisely those of the form \(\tbasis'_i - \tbasis''_j\).  If \ul r and \ol r are the greatest common divisor and least common multiple of \(r'\) and \(r''\), respectively, then there are \(2\ul r\) orbits of such roots, each of size \(2\ol r\); they are parameterised by elements of \(\Z/2\ul r\Z\), and the roots of the form \(\tbasis'_0 - \tbasis''_i\) lying in the orbit corresponding to \(j\) are precisely those for which \(i \equiv j \pmod{2\ul r}\).
The element \(w_{\Root'}\) sends the orbit (in \(\Root \cap (\Q\Root' \oplus \Q\Root'')\), but not in \(\Q\Root'\) or \(\Q\Root''\)) parameterised by \(j\) to the one parameterised by \(j - 1\), hence acts on the orbits in \(\Root \cap (\Q\Root' \oplus \Q\Root'')\) that are not in \(\Q\Root'\) or \(\Q\Root''\) by a \(2\ul r\)-cycle; in particular, by an odd permutation.  The element \(\norm_{\Root', q}\) sends \(\tbasis'_0 - \tbasis''_j\) to \(\tbasis'_0 - \tbasis''_{q j}\), hence, by Proposition \ref{prop:Legendre}, acts on the same set of orbits as a permutation of sign \Legend q{2\ul r}.  In particular, if \ul r is odd (\ie, if \(r'\) or \(r''\) is odd), then the action is by an even permutation; and, if \ul r is even (\ie, if both \(r'\) and \(r''\) are even), then it is by a permutation of sign \((-1)^{(q - 1)/2}\).

If \(r' = r''\), then there is a unique involution \(v = v_{\Root', \Root''} \in C_A(w)\) that fixes pointwise every irreducible component of \Cn[\lambda] other than \(\Root'\) and \(\Root''\), and satisfies \(v\tbasis'_i = \tbasis''_i\) for all \(i \in \Z/2r'\Z\).  This involution sends the orbit (in \(\Root \cap (\Q\Root' \oplus \Q\Root'')\), but not in \(\Q\Root'\) or \(\Q\Root''\)) parameterised by \(j\) to the one parameterised by \(-j\).  In particular, its fixed points in this space of orbits are precisely the orbits of \(\tbasis'_0 - \tbasis''_0 = \basis'_0 - \basis''_0\) and of \(\tbasis'_0 - \tbasis''_{r'} = \basis'_0 + \basis''_0\).  Since it has no fixed points in the spaces of orbits in \(\Root \cap \Q\Root'\) or \(\Root \cap \Q\Root''\), it acts on the space of orbits in \(\Root \cap (\Q\Root' \oplus \Q\Root'')\) by \(\frac1 2(r' + r' + (2r' - 2)) = 2r' - 1\) transpositions in the non-simply laced cases, and by \(\frac1 2((r' - 1) + (r' - 1) + (2r' - 2)) = 2r' - 2\) transpositions in the simply laced case, hence as an odd, respectively even, permutation.

Now let \(\Root'''\) be a third (still different) irreducible component of \Cn[\lambda].  Then the fixed points of the action of \(v\) on the space of orbits in \(\Root \cap (\Q\Root' \oplus \Q\Root'' \oplus \Q\Root'')\) lie in \(\Q\Root' \oplus \Q\Root''\) or \(\Q\Root'''\).

We have shown that
\begin{align}
\label{eq:BCD:Coxeter}
\Legend{w_{\Root'}}w
& {}= (-1)^{s - 1}, \\
\label{eq:BCD:switch}
\Legend{v_{\Root', \Root''}}w
& {}= \begin{cases}
-1, & \text{in the non-simply laced cases,} \\
1,  & \text{in the simply laced case,}
\end{cases}
\intertext{and}
\label{eq:BCD:norm}
\Legend{\norm_{\Cn[\lambda], q}}w
& {}= (-1)^{s'(s' - 1)(q - 1)/4}
	\prod_{\substack
		{\operatorname{rk} \Root' = r' \\
		\text{\(r'\) even}}
	}
		\Legend{2r'}q
\end{align}
where \(s\) is the number of irreducible components of \Cn[\lambda], \(s'\) is the number of those components of even rank, and the product in \eqref{eq:BCD:norm} runs over the components \(\Root'\) of even rank \(r'\).  In particular,
\[
C_{W(\Dn)}(w) \subseteq \ker \Legend\cdot w
\]
if and only if \Cn[\lambda] decomposes with multiplicity \(1\), or \(\Root = \Dn\); and
\[
C_A(w) \subseteq \ker \Legend\cdot w
\]
if and only if, in addition, \(s\) is odd, \ie, \(w \not\in W(\Dn)\).

Recall that, if \(\Root = \Cn\), then all parts of \(\lambda\) are odd, so \(s' = 0\) and
\[
\Legend{\norm_{\Cn[\lambda], q}}w = 1;
\]
or all parts of \(\lambda\) are even (hence also \(n\) is even).

\section{Type \texorpdfstring{$\mathsf F_4$}{F4}}
\label{sec:F4}

The results of this section are summarised in Table \ref{tab:F4}.

The root system \Fn contains a copy of \Bn[4].
If \(G_{\Fn}\) is the semisimple group whose root system (with respect to some maximal torus \(T_{\Fn}\)) is of type \Fn, then, since \(G_{\Fn}\) is simply connected, its subgroup generated by \(T_{\Fn}\) and the root subgroups corresponding to roots in \(\Bn[4]\) is also simply connected, hence of type \(\Spin(V_{\Bn[4]})\).

If \((\root_1 = \basis_0 - \basis_1, \root_2 = \basis_1 - \basis_2, \root_3 = \basis_2 - \basis_3, \root_4 = \basis_3)\)
\[\xymatrix{
\root_1 \ar@{-}[r] & \root_2 \ar@{-}[r] & \root_3 \ar@{=>}[r] & \root_4
}\]
is a system of simple roots in \Bn[4], in the ordering of \bourbaki*{Planche VI.II}, then \(\bigl(\root_2, \root_3, \root_4, \frac1 2(\root_1 - \root_3 - 2\root_4)\bigr)\)
\[\xymatrix{
\root_2 \ar@{-}[r] & \root_3 \ar@{=>}[r] & \root_4 \ar@{-}[r] & \frac1 2(\root_1 - \root_3 - 2\root_4)
}\]
is a system of simple roots for \Fn, in the ordering of \bourbaki*{Planche VI.VII}.
\def\dB{\me{\dot{\type B}_4}}%
\def\dF{\me{\dot{\type F}_4}}%
\def\Fperp{\me{\unwrap\Fn^\perp}}%
\def\dFperp{\me{\dot{\type F}_4^\perp}}%
Put \(\Fperp = \Fn \setminus \Bn[4]\).

Put \(A = W = W(\Fn)\), and let \(w\) be an elliptic element of \Fn.

\subsection{\texorpdfstring{$\mathsf A_3 \times \widetilde{\mathsf A}_1$}{A3 x A1}}
\label{sec:A3xA1-in-F4}

The element \(w \ldef w_{\Bn[2 + 1^{+2}]} = \ncycle(0)\ncycle(1)\ncycle(2 3) \in W(\Bn[4])\) is attached to the admissible diagram
\[\xymatrix{
\basis_0 & 
\basis_1 & 
\root_3 \ar@{=}[r] & \root_4
}\]
of type \(\Bn[2] \times \tilde{\type B}_1^2\), but also, since \(\refl_{\basis_1}\refl_{\root_4} = \refl_{\basis_1 - \basis_3}\refl_{\basis_1 + \basis_3}\),
to the admissible diagram
\[\xymatrix{
\basis_0 & \basis_1 - \basis_3 \ar@{-}[r] & 
\root_3 \ar@{-}[r] &
\basis_1 + \basis_3 & 
}\]
of type \(\An[3] \times \tilde{\type A}_1\).  By Remark \ref{rem:char-values} and the explicit description of \Fn in \bourbaki*{Planche VI.VIII}, the element \(\dot w_{\Bn[2 + 1^{+2}]}\) acts on \(\Lie(G_{\Fn})/\spin(V_{\Bn[4]})\) with eigenvalues precisely the primitive \(8\)th roots of unity, each occurring with multiplicity \(4\).  In particular, \((\Fn)_w \setminus (\Bn[4])_w\) is empty, so \((\Fn)_w = (\Bn[4])_w = \Bn[4]\) by Remark \ref{rem:Rw-full}.

The group \(C_{W(\Bn[4])}(w)\) is
\[
\bigl(
	(\sgen{\ncycle(0)} \times \sgen{\ncycle(1)}) \rtimes
	\sgen{\cycle(0 1)}
\bigr) \times
\sgen{\ncycle(2 3)},
\]
which has order \(32\).  By \carter*{Table 8}, it equals \(C_{W(\Fn)}(w)\); so also \(N_{W(\Fn)}(\sgen w) = C_{W(\Fn)}(w) \rtimes \smashed\sgen{\norm_{\Cn[2 + 1^{+2}], -1}} = N_{W(\Bn[4])}(\sgen w)\).  Thus, \(\Legend\cdot w_{\Fn} = \Legend\cdot w_{\Bn[4]}\).

\subsection{\texorpdfstring{$\mathsf D_4(a_1)$}{D4(a1)}}
\label{sec:D4(a1)-in-F4}

By Remark \ref{rem:Dn-as}, the element \(w \ldef w_{\Bn[2^{+2}]} = \ncycle(0 1)\ncycle(2 3) \in W(\Bn[4])\) is attached to the admissible diagram
\[\xymatrix{
& \root_2 \ar@{-}[dr] \\
\root_1 \ar@{-}[ur]\ar@{-}[dr] & & \root_3 \\
& \basis_1 + \basis_2 \ar@{-}[ur] 
}\]
of type \(\Dn[4](a_1)\).  By Remark \ref{rem:char-values}, the element \(\dot w_{\Bn[2^{+2}]}\) acts on \(\Lie(G_{\Fn})/\spin(V_{\Bn[4]})\) with eigenvalues precisely the \(4\)th roots of unity, each occurring with multiplicity \(4\).  In particular, by dimension counting (since each orbit of \(w\) in \((\Fn)_w\) contributes \(1\) to the dimension of the space of \(\dot w_{\Bn[2^{+2}]}\)-fixed points in \(\Lie(G_{\Fn})/\spin(V_{\Bn[4]})\)), we have that \(\card{\sgen w\bslash((\Fn)_w \setminus (\Bn[4])_w)} = 4\).  The element \(w\) has \(4\) orbits on \Fperp, each of which is symmetric and has \(4\) elements.  Thus, since \((\Bn[4])_w = \Bn[4]\) by Remark \ref{rem:Rw-full}, we have \((\Fn)_w = \Fn\).  Put \(\dB = \sgen w\bslash\Bn[4]\) and \(\dFperp = \sgen w\bslash\Fperp\).

The group \(C_{W(\Bn[4])}(w)\) is
\[
(\sgen{\ncycle(0 1)} \times \sgen{\ncycle(2 3)}) \rtimes \sgen{\cycle(0 2)\cycle(1 3)},
\]
which has order \(2^5\).  By \carter*{Table 8}, it is a Sylow \(2\)-subgroup of \(C_{W(\Fn)}(w)\).
(In fact it has index \(3\) in \(C_{W(\Fn)}(w)\), and the full centraliser has a \(3\)-Sylow subgroup generated by an element of \(N_{W(\Fn)}(\Dn[4]) \setminus W(\Bn[4])\).)  The element \ncycle(0 1) acts transitively on \dFperp; whereas the element \(\cycle(0 2)\cycle(1 3)\) has \(1\) orbit of size \(2\), and \(2\) orbits of size \(1\) (namely, those through \(\frac1 2(\root_1 \pm \root_3)\)).  Thus, both act with sign \(-1\) on \dFperp.  By \eqref{eq:BCD:Coxeter} and \eqref{eq:BCD:switch}, they also act with sign \(-1\) on \dB.  Thus, each of a set of normal generators of \(C_{W(\Bn[4])}(w)\), which is a Sylow \(2\)-subgroup of \(C_{W(\Fn)}(w)\), acts with sign \(1\) on \dF; so every element of \(C_{W(\Fn)}(w)\) acts with sign \(1\) on \dF.

The group \(N_{W(\Fn)}(\sgen w)\) is the semi-direct product of \(C_{W(\Fn[4])}(w)\) and the group generated by \(\norm_{\Bn[2^{+2}], -1} = \ncycle(1)\ncycle(3) \in W(\Bn[4])\).  Since \(\norm_{\Bn[2^{+2}], -1}\cycle(0 2)\cycle(1 3)\) acts trivially on \dFperp, the element \(\norm_{\Bn[2^{+2}], -1}\) acts with sign \(-1\) on \dFperp.  By \eqref{eq:BCD:norm}, it acts with sign
\[
(-1)^{2(2 - 1)(-1 - 1)/4}\Legend{2\cdot2}{-1}^2 = -1
\]
on \dB, hence with sign \(1\) on \dF.

\subsection{\texorpdfstring{$\mathsf B_4$}{B4}}
\label{sec:B4-in-F4}

The element \(w = w_{\Bn[4]} = \ncycle(0 1 2 3) \in W(\Bn[4])\) is a Coxeter element of \Bn[4]; in particular, it is attached to an admissible diagram of type \Bn[4].  By Remark \ref{rem:char-values}, the element \(\dot w_{\Bn[4]}\) acts on \(\Lie(G_{\Fn})/\spin(V_{\Bn[4]})\) with eigenvalues precisely the \(8\)th roots of unity, each occurring with multiplicity \(2\).  In particular, by dimension counting, \(\card{\sgen w\bslash((\Fn)_w \setminus (\Bn[4])_w)} = 2\).
The element \(w\) has \(2\) orbits on \Fperp, each of which is symmetric and has \(8\) elements; namely, those through
\(\frac1 2(\root_1 + \root_3)\) 
and \(\frac1 2(\root_1 + 2\root_2 + \root_3)\).  
Thus, since \((\Bn[4])_w = \Bn[4]\) by Remark \ref{rem:Rw-full}, we have \((\Fn)_w = \Fn\).  Put \(\dB = \sgen w\bslash\Bn[4]\) and \(\dFperp = \sgen w\bslash\Fperp\).

The group \(C_{W(\Bn[4])}(w)\) is generated by \(w\) \cite{springer:regular}*{Corollary 4.4}, hence has order \(2^3\).  By \carter*{Table 8}, it equals \(C_{W(\Fn)}(w)\).

The group \(N_{W(\Fn)}(\sgen w)\) is the semi-direct product of \(C_{W(\Fn)}(w)\) and \(\sgen{\norm_{\Bn[4], 3}} \times \sgen{\norm_{\Bn[4], -3}} \subseteq W(\Bn[4])\).  The elements \(\norm_{\Bn[4], 3} = \cycle(1 3)\ncycle(2)\) and \(\norm_{\Bn[4], -3} = \ncycle(1)\ncycle(3)\) swap the two orbits outside of \Bn[4], hence act with sign \(-1\) on \dFperp.
By \eqref{eq:BCD:norm}, we have that \(\norm_{\Bn[4], q}\) acts with sign
\[
(-1)^{1(1 - 1)(q - 1)/4}\Legend{2\cdot4}q = \Legend8 q = -1
\]
on \dB, hence also on \dF, for \(q \in \sset{\pm3}\).

\section{Types \texorpdfstring{$\mathsf E_n$}{En}}
\label{sec:En}

The results of this section are summarised in Tables \ref{tab:E8}--\ref{tab:E6}.

The root system \En[8] contains a copy of \Dn[8].
If \(G_{\En[8]}\) is the semisimple group whose root system (with respect to a maximal torus \(T_{\En[8]}\)) is \En[8], then, since \(G_{\En[8]}\) is simply connected, its subgroup generated by \(T_{\En[8]}\) and the root subgroups corresponding to roots in \Dn[8] is also simply connected, hence of type \(\Spin(V_{\Dn[8]})\).

If \((\root_1 = -(\basis_6 + \basis_7), \root_2 = \basis_6 - \basis_5, \root_3 = \basis_5 - \basis_4, \root_4 = \basis_4 - \basis_3, \root_5 = \basis_3 - \basis_2, \root_6 = \basis_2 - \basis_1, \root_7 = \basis_1 - \basis_0, \root_8 = \basis_0 + \basis_1)\)
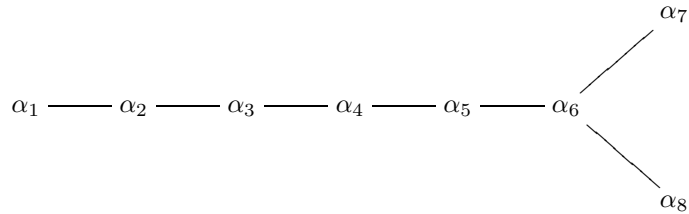
\[\xymatrix{
& & & & & & \root_7 \\
\root_1 \ar@{-}[r] & \root_2 \ar@{-}[r] & \root_3 \ar@{-}[r] & \root_4 \ar@{-}[r] & \root_5 \ar@{-}[r] & \root_6 \ar@{-}[ur]\ar@{-}[dr] \\
& & & & & & \root_8
}\]
is a system of simple roots in \Dn[8], in the ordering of \bourbaki*{Planche VI.IV}, then \(\bigl(\otherroot \ldef -\frac1 2(\root_1 + 2\root_2 + 3\root_3 + 4\root_4 + 5\root_5 + 6\root_6 + 4\root_7 + 3\root_8), \root_8, \root_7, \root_6, \root_5, \root_4, \root_3, \root_2\bigr)\)
\[\xymatrix{
\otherroot \ar@{-}[r] & \root_7 \ar@{-}[r] & \root_6 \ar@{-}[d]\ar@{-}[r] & \root_5 \ar@{-}[r] & \root_4 \ar@{-}[r] & \root_3 \ar@{-}[r] & \root_2 \\
& & \root_8
}\]
is a system of simple roots for \En[8], in the ordering of \bourbaki*{Planche VI.VII}.

Write
\(\root_0 = -(\root_1 + 2\root_2 + 2\root_3 + 2\root_4 + 2\root_5 + 2\root_6 + \root_7 + \root_8)\) 
for the lowest root of \Dn[8].

The orthogonal complement of \(\root_1\) in \Dn[8] (respectively, \En[8]) is a root system of type \(\Dn[6] \times \Cn[1]\) (respectively, \En[7]).  
The common orthogonal complement of \(\root_1\) and \(\root_2\) in \Dn[8] (respectively, \En[8]) is a root system of type \Dn[5] (respectively, \En[6]).
\def\dd{\me{\dot{\type D}_5}}%
\def\dDC{\me{(\Dn[6] \times \Cn[1])\spdot}}%
\def\dD{\me{\dot{\type D}_8}}%
\newcommand\dEn[1][n]{\me{\dot{\type E}_{#1}}}%
\newcommand\Enperp[1][n]{\me{\unwrap\En[#1]^\perp}}%
\newcommand\dEnperp[1][n]{\me{\dot{\type E}_{#1}^\perp}}%
Put \(\Enperp[8] = \En[8] \setminus \Dn[8]\), \(\Enperp[7] = \En[7] \setminus (\Dn[6] \times \Cn[1])\), and \(\Enperp[6] = \En[6] \setminus \Dn[5]\).

\subsection{\texorpdfstring{$-\mathsf D_4(a_1)$, $\mathsf A_3^2 \times \mathsf A_1$, and $\mathsf A_3^2 \times \mathsf A_1^2$}{-D4(a1), A32 x A1, and A32 x A12}}
\label{sec:A3xA3xA1xA1-in-E8}

Let \mc D be the admissible diagram
\[\xymatrix{
& \root_3 & \root_7 \ar@{-}[dr] \\
\root_4 \ar@{-}[ur]\ar@{-}[dr] & & & \root_6 & \root_0 \\
& \basis_4 + \basis_5 & 
\root_8 \ar@{-}[ur]
}\]
in \(\Dn[6] \times \Cn[1]\).  The conjugacy class attached to \(\mc D \sqcup \sset{\root_1}\) of type \(\unwrap\Dn[3]^2 \times \unwrap\Cn[1]^2 = \unwrap\An[3]^2 \times \unwrap\An[1]^2\) is a product of Coxeter classes in
\begin{itemize}
\item \(\Dn[8] \cap \Z\sset{\root_3, \root_4, \basis_4 + \basis_5} = \sett{\pm\basis_i \pm \basis_j}{\(i, j \in \sset{3, 4, 5}\) and \(i \ne j\)}\),
\item \(\Dn[8] \cap \Z\sset{\root_6, \root_7, \root_8} = \sett{\pm\basis_i \pm \basis_j}{\(i, j \in \sset{0, 1, 2}\) and \(i \ne j\)}\),
and
\item \sset{\pm\root_0 = \pm(\basis_7 - \basis_6)} and \sset{\pm\root_1 = \pm(\basis_6 + \basis_7)}.
\end{itemize}
In particular, it contains the element \(w_8 \ldef w_{\Dn[2^{+2} + 1^{+4}]} = \ncycle(0)\ncycle(1 2)\ncycle(3 4)\ncycle(5)\ncycle(6)\ncycle(7)\).

Since \(w_8\) negates \(\root_1\) and \(\root_2\), it stabilises \Dn[5], \En[6], \(\Dn[6] \times \Cn[1]\), and \En[7].  Its restriction \(w_7 = w_{\Dn[2^{+2} + 1^{+2}]}w_{\Cn[1]} \in W(\Dn[6] \times \Cn[1])\) to \En[7] is attached to the admissible diagram \mc D, which is of type \(\unwrap\Dn[3]^2 \times \Cn[1] = \unwrap\An[3]^2 \times \An[1]\).  Write \(w_6\) for its restriction to \En[6].  By Remark \ref{rem:Dn-as}, the element \(-w_6 = \ncycle(2 1)\ncycle(4 3)\) is attached to the admissible diagram
\[\xymatrix{
& \basis_3 - \basis_1 \ar@{-}[dr] \\ 
\root_4 \ar@{-}[ur]\ar@{-}[dr] & & \root_6 \\
& \basis_1 + \basis_3 \ar@{-}[ur] 
}\]
of type \(\Dn[4](a_1)\).

By Remark \ref{rem:char-values} and the explicit description of \En[8] in \bourbaki*{Planche VI.VII}, the element \(\dot w_{\Dn[2^{+2} + 1^{+4}]}\) acts on \(\Lie(G_{\En[8]})/\spin(V_{\Dn[8]})\) with eigenvalues precisely the \(4\)th roots of unity, each occurring with multiplicity \(32\).  In particular, by dimension counting, \(\card{\sgen{w_8}\bslash((\En[8])_{w_8} \setminus (\Dn[8])_{w_8})} = 32\).  The element \(w_8\) has \(32\) orbits on \Enperp[8], each of size \(4\), and none symmetric.  Thus, since \((\Dn[8])_{w_8} = \Dn[8]\) by Remark \ref{rem:Rw-full}, we have \((\En[8])_{w_8} = \En[8]\), hence also \((\En)_{w_n} = \En\) for \(n \in \sset{6, 7}\).  Of the \(32\) orbits of \(w_8\) on \Enperp[8], \(16\) lie in \Enperp[7] (hence are orbits of \(w_7\) there), and \(8\) in \Enperp[6] (hence are orbits of \(w_6\) there).  Put \(\dd = \sgen{w_6}\bslash\Dn[5]\), \(\dDC = \sgen{w_7}\bslash(\Dn[6] \times \Cn[1])\), \(\dD = \sgen{w_8}\bslash\Dn[8]\), and \(\dEnperp = \sgen{w_n}\bslash\Enperp\) for \(n \in \sset{6, 7, 8}\).

The group \(C_{W(\Dn[8])}(w_8)\) is the kernel of \(\sgn\textsub{\Bn[8], short}\) on
\begin{multline*}
\bigl((\sgen{\ncycle(0)} \times \sgen{\ncycle(5)} \times \sgen{\ncycle(6)} \times \sgen{\ncycle(7)}) \rtimes \sgen{\cycle(0 5), \cycle(5 6), \cycle(6 7)}\bigr) \times \\
\bigl((\sgen{\ncycle(1 2)} \times \sgen{\ncycle(3 4)}) \rtimes \sgen{\cycle(1 3)\cycle(2 4)}\bigr).
\end{multline*}
Since \(\sgen{\cycle(0 5), \cycle(5 6), \cycle(6 7)} \cong \Sgp_4\), the order of this group is \(2^{11}\dotm3\).  By \carter*{Table 11}, it contains a Sylow \(2\)-subgroup of \(C_{W(\En[8])}(w_8)\).
The elements \(\ncycle(0)\ncycle(5)\) (and its \(C_{W(\Dn[8])}(w_8)\)-conjugate \(\ncycle(6)\ncycle(7)\)), \(\ncycle(0)\ncycle(1 2)\), \cycle(0 5), and \(\cycle(1 3)\cycle(2 4)\) have \(16\) orbits of size \(2\), \(8\) orbits of size \(4\), \(24\) orbits (of which \(8\) have size \(2\) and \(16\) have size \(1\)), and \(24\) orbits (of which \(8\) have size \(2\) and \(16\) have size \(1\)) on \dEnperp[8], none of them symmetric, hence act with sign \(1\) there.  By \eqref{eq:BCD:Coxeter} and \eqref{eq:BCD:switch}, \(\ncycle(0)\ncycle(5)\), \(\ncycle(0)\ncycle(1 2)\), \cycle(0 5), and \(\cycle(1 3)\cycle(2 4)\) all act with sign \(1\) on \dD.  Thus, each of a set of normal generators of \(C_{W(\Dn[8])}(w_8)\), which contains a Sylow \(2\)-subgroup of \(C_{W(\En[8])}(w_8)\), acts with sign \(1\) on \dEn[8]; so every element of \(C_{W(\En[8])}(w_8)\) acts with sign \(1\) on \dEn[8].

The group \(C_{W(\Dn[6] \times \Cn[1])}(w_7)\) is the intersection with \(W(\Dn[6] \times \Cn[1])\) of \(C_{W(\Dn[8])}(w_8)\); explicitly, it is the direct product with \sgen{\ncycle(6)\ncycle(7)} of \(\sgn\textsub{\Bn[6], short}\) on
\[
\bigl((\sgen{\ncycle(0)} \times \sgen{\ncycle(5)}) \rtimes \sgen{\cycle(0 5)}\bigr) \times
\bigl((\sgen{\ncycle(1 2)} \times \sgen{\ncycle(3 4)}) \rtimes \sgen{\cycle(1 3)\cycle(2 4)}\bigr).
\]
It has order \(2^8\), hence, by \carter*{Table 10}, is a Sylow \(2\)-subgroup of \(C_{W(\En[7])}(w_7)\).
We showed above that \(\ncycle(0)\ncycle(5)\), \cycle(0 5), \(\ncycle(6)\ncycle(7))\), \(\ncycle(0)\ncycle(1 2)\), and \(\cycle(1 3)\cycle(2 4)\) act without symmetric orbits, hence with sign \(1\), on \dEnperp[7] (indeed, on \dEnperp[8]).  Since \(C_{W(\Dn[6] \times \Cn[1])}(w_7)\) acts trivially on the (\(2\)-element) orbit of \(w_7\) in \Cn[1], we have by \eqref{eq:BCD:Coxeter} and \eqref{eq:BCD:switch} that \(\ncycle(0)\ncycle(5)\), \(\ncycle(0)\ncycle(1 2)\), \cycle(0 5), and \(\cycle(1 3)\cycle(2 4)\) all act with sign \(1\) on \dDC.  The element \(\ncycle(6)\ncycle(7)\) acts trivially on \Dn[6], hence also acts with sign \(1\) on \dDC.  Thus, each of a set of normal generators of the Sylow \(2\)-subgroup \(C_{W(\Dn[6] \times \Cn[1])}(w_7)\) of \(C_{W(\En[7])}(w_7)\) acts with sign \(1\) on \dEn[7]; so every element of \(C_{W(\En[7])}(w_7)\) acts with sign \(1\) on \dEn[7].

The centraliser in \(N_{W(\En[6])}(\Dn[5])\) of \(w_6\) is the intersection with \(N_{W(\En[6])}(\Dn[5])\) of \(C_{W(\Dn[8])}(w_8)\); explicitly, it is the kernel of the homomorphism
\[
\anonmap
	{\sgen{\ncycle(0)} \times
	\sgen{\ncycle(5)\ncycle(6)\ncycle(7)} \times
	\bigl((\sgen{\ncycle(1 2)} \times \sgen{\ncycle(3 4)}) \rtimes \sgen{\cycle(1 3)\cycle(2 4)}\bigr)}
	{\sgen{-1}}
\]
sending each indicated product of negative cycles to \(-1\), and \(\cycle(1 3)\cycle(2 4)\) to \(1\).  It has order \(2^6\), hence, by \carter*{Table 9}, is a Sylow \(2\)-subgroup of \(C_{\Aut(\En[6])}(w_6) = C_{W(\En[6])}(w_6) \times \sgen{-1}\).
We showed above that \(\ncycle(0)\ncycle(1 2)\) and \(\cycle(1 3)\cycle(2 4)\) act without symmetric orbits, hence with sign \(1\), on \dEnperp[6] (indeed, on \dEnperp[8]).  The element \(\ncycle(0)\ncycle(5)\ncycle(6)\ncycle(7)\) has \(4\) orbits of size \(2\), all symmetric, on \dEnperp[6], hence acts there with sign \(1\); and it acts trivially on \Dn[5].  Viewed as an element of \(\Aut(\Dn[5])\), the element \(w_6\) is of the form \(w_{\Cn[2^{+2} + 1]}\); and, by \eqref{eq:BCD:Coxeter} and \eqref{eq:BCD:switch}, \(\ncycle(0)\ncycle(1 2)\) and \(\ncycle(0)\ncycle(5)\ncycle(6)\ncycle(7)\) act with sign \(1\) on \dd.  On the other hand, \(\ncycle(0)\ncycle(5)\ncycle(6)\ncycle(7)\) negates each orbit in \En[6] outside \Dn[5].  Thus, each of a set of normal generators of \(C_{N_{W(\En[6])}(\Dn[5])}(w_6)\), which is a Sylow \(2\)-subgroup of \(C_{\Aut(\En[6])}(w_6)\), act with sign \(1\) on \dEn[6]; so every element of \(C_{\Aut(\En[6])}(w_6)\) acts with sign \(1\) on \dEn[6].

For \(n \in \sset{6, 7, 8}\), the group \(N_{\Aut(\En)}(\sgen{w_n})\) is the semi-direct product of \(C_{\Aut(\En)}(w_n)\) and the group generated by \(\norm_{\Cn[2^{+2} + 1^{+4}], -1} = \ncycle(2)\ncycle(4) \in W(\Dn[5])\), which element has \(16\) orbits on \dEnperp[8], all of size \(2\), and \(8\) of them non-symmetric.  The symmetric orbits in \dEnperp[6] are those through \(\frac1 2(\root_1 \pm \root_3 + \root_5 - \root_7)\); the additional symmetric orbits in \dEnperp[7] are those through \(\frac1 2(\root_1 \pm \root_3 + \root_5 + \root_7)\); and the remaining symmetric orbits in \dEnperp[8] are those through \(\frac1 2(\root_1 \pm \root_3 + \root_5 + 2\root_6 + \root_7)\) and \(\frac1 2(\root_1 \pm \root_3 + \root_5 + 2\root_6 + \root_7 + 2\root_8)\).
That is, \(\norm_{\Cn[2^{+2} + 1^{+4}], -1}\) acts with sign \(1\) on \dEnperp for \(n \in \sset{6, 7, 8}\).  The transformation \(\norm_{\Cn[2^{+2} + 1^{+4}], -1}\) of \dD acts on \dDC as \(\norm_{\Cn[2^{+2} + 1^{+2}], -1}\) and on \dd as \(\norm_{\Cn[2^{+2} + 1], -1}\); so, by \eqref{eq:BCD:norm}, it acts in each case with sign
\[
(-1)^{2(2 - 1)(-1 - 1)/4}\Legend{2\cdot2}{-1}^2 = -1,
\]
hence with the same sign on \dEn for \(n \in \sset{6, 7, 8}\).

\subsection{\texorpdfstring{$\mathsf A_7$ and $\mathsf A_7 \times \mathsf A_1$}{A7 and A7 x A1}}
\label{sec:A7xA1-in-E8}

Let \mc D be the admissible diagram
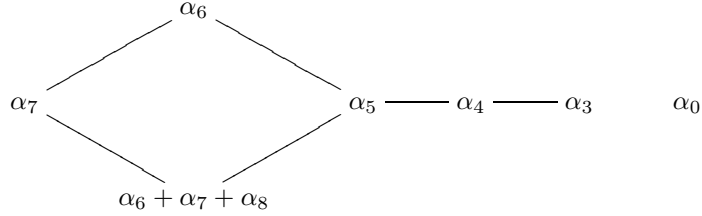
\[\xymatrix{
& \root_6 \ar@{-}[dr] \\
\root_7 \ar@{-}[ur]\ar@{-}[dr] & & \root_5 \ar@{-}[r] & \root_4 \ar@{-}[r] & \root_3 & \root_0 \\
& \root_6 + \root_7 + \root_8 \ar@{-}[ur]
}\]
in \(\Dn[6] \times \Cn[1]\) of type \(\Dn[6](a_1) \times \Cn[1]\).  By Remark \ref{rem:Dn-as}, the element \(w_8 \ldef w_{\Dn[4 + 2 + 1^{+2}]} = \ncycle(0 1)\ncycle(2 3 4 5)\ncycle(6)\ncycle(7) \in W(\Dn[8])\) is attached to the admissible diagram \(\mc D \sqcup \sset{\root_1}\) of type \(\Dn[6](a_1) \times \unwrap\Cn[1]^2\).

Since \(w_8\) negates \(\root_1\), it stabilises \(\Dn[6] \times \Cn[1]\) and \En[7].  Its restriction \(w_7 = w_{\Dn[4 + 2]}w_{\Cn[1]}\in W(\Dn[6] \times \Cn[1])\) to \En[7] is attached to the admissible diagram \mc D.  On the other hand, since
\[
\refl_{\root_0}\refl_{\root_3}\refl_{\root_5}\refl_{\root_7} = \refl_{\otherroot_1}\refl_{\otherroot_2}\refl_{\otherroot_3}\refl_{\otherroot_4},
\]
where
\[
\begin{pmatrix}
\otherroot_1 \\ 
\otherroot_2 \\ 
\otherroot_3 \\ 
\otherroot_4    
\end{pmatrix} = \frac1 2\begin{pmatrix}
1 & 1  & -1 & -1 \\
1 & 1  & 1  & 1  \\
1 & -1 & -1 & 1  \\
1 & -1 & 1  & -1
\end{pmatrix}\begin{pmatrix}
\root_0 \\
\root_3 \\
\root_5 \\
\root_7
\end{pmatrix},
\]
we have that \(w_8\), viewed as an element of \(W(\En[8])\) (respectively, \(w_7\), viewed as an element of \(W(\En[7])\)), is also attached to the admissible diagram \(\mc D' \sqcup \sset{\root_1}\) of type \(\An[7] \times \An[1]\) (respectively, to the admissible diagram \(\mc D'\) of type \An[7]), where \(\mc D'\) is the admissible diagram
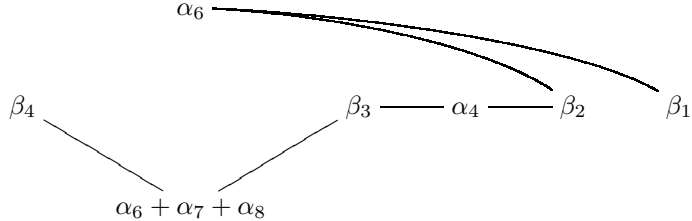
\[\xymatrix{
& \root_6 \ar@(r,ul)@{-}[drrr]\ar@(r,ul)@{-}[drrrr] \\
\otherroot_4 \ar@{-}[dr] & & \otherroot_3 \ar@{-}[r] & \root_4 \ar@{-}[r] & \otherroot_2 & \otherroot_1 \\
& \root_6 + \root_7 + \root_8 \ar@{-}[ur]
}\]
in \En[7].

By Remark \ref{rem:char-values}, the element \(\dot w_{\Dn[4 + 2 + 1^{+2}]}\) acts on \(\Lie(G_{\En[8]})/\spin(V_{\Dn[8]})\) with eigenvalues precisely the \(8\)th roots of unity, each occurring with multiplicity \(16\).  In particular, by dimension counting, \(\card{\sgen{w_8}\bslash((\En[8])_{w_8} \setminus (\Dn[8])_{w_8})} = 16\).  The element \(w_8\) has \(16\) orbits on \Enperp[8], each of size \(8\), and none symmetric.  Thus, since \((\Dn[8])_{w_8} = \Dn[8]\) by Remark \ref{rem:Rw-full}, we have \((\En[8])_{w_8} = \En[8]\), hence also \((\En[7])_{w_7} = \En[7]\).  Of the \(16\) orbits of \(w_8\) on \Enperp[8], \(8\) lie in \Enperp[7] (hence are orbits of \(w_7\) there).  Put \(\dD = \sgen{w_8}\bslash\Dn[8]\), \(\dDC = \sgen{w_7}\bslash(\Dn[6] \times \Cn[1])\), and \(\dEnperp = \sgen{w_n}\bslash\Enperp\) for \(n \in \sset{7, 8}\).

The group \(C_{W(\Dn[8])}(w_8)\) is the kernel of \(\sgn\textsub{\Bn[8], short}\) on
\[
\sgen{\ncycle(0 1)} \times
\sgen{\ncycle(2 3 4 5)} \times
\bigl((\sgen{\ncycle(6)} \times \sgen{\ncycle(7)}) \rtimes \sgen{\cycle(6 7)}\bigr).
\]
It has order \(2^7\), so, by \carter*{Table 11}, equals \(C_{W(\En[8])}(w_8)\).  The elements \(\ncycle(0 1)\ncycle(2 3 4 5)\), \(\ncycle(0 1)\ncycle(6)\), and \cycle(6 7) have \(8\) orbits of size \(2\), \(4\) orbits of size \(4\), and \(12\) orbits (of which \(4\) have size \(2\) and \(8\) have size \(1\)), respectively, on \dEnperp[8], none of them symmetric.  By \eqref{eq:BCD:Coxeter} and \eqref{eq:BCD:switch}, the elements \(\ncycle(0 1)\ncycle(2 3 4 5)\), \(\ncycle(0 1)\ncycle(6)\), and \cycle(6 7) all act with sign \(1\) on \dD.  Thus, each of a set of normal generators of \(C_{W(\En[8])}(w_8)\) acts with sign \(1\) on \dEn[8]; so every element of \(C_{W(\En[8])}(w_8)\) acts with sign \(1\) on \dEn[8].

The group \(C_{W(\Dn[6] \times \Cn[1])}(w_7)\) is the intersection with \(W(\Dn[6] \times \Cn[1])\) of \(C_{W(\Dn[8])}(w_8)\); explicitly, it is the direct product with \sgen{\cycle(6 7)} of the kernel of \(\sgn\textsub{\Bn[6], short}\) on
\[
\sgen{\ncycle(0 1)} \times \sgen{\ncycle(2 3 4 5)}.
\]
It has order \(2^5\), hence, by \carter*{Table 10}, equals \(C_{W(\En[7])}(w_7)\).  We showed above that \(\ncycle(0 1)\ncycle(2 3 4 5)\) and \(\cycle(6 7)\) act without symmetric orbits, hence with sign \(1\), on \dEnperp[7]; and \(\ncycle(0 1)^2\) acts trivially, hence with sign \(1\), on \dEnperp[7] (indeed, on \dEnperp[8]).  Since \(C_{W(\Dn[6] \times \Cn[1])}(w_7)\) acts trivially on the (\(2\)-element) orbit of \(w_7\) in \Cn[1], we have by \eqref{eq:BCD:Coxeter} that \(\ncycle(0 1)^2\) and \(\ncycle(0 1)\ncycle(2 3 4 5)\) act with sign \(1\) on \dDC.  The element \cycle(6 7) acts trivially on \Dn[6], hence also acts with sign \(1\) on \dDC.  Thus, each of a set of normal generators of \(C_{W(\En[7])}(w_7)\) acts with sign \(1\) on \dEn[7]; so every element of \(C_{W(\En[7])}(w_7)\) acts with sign \(1\) on \dEn[7].

For \(n \in \sset{7, 8}\), the group \(N_{W(\En)}(\sgen{w_n})\) is the semi-direct product of \(C_{W(\En)}(w_n)\) and \(\smashed\sgen{\norm_{\Dn[4 + 2 + 1^{+2}], 3}} \times \smashed\sgen{\norm_{\Dn[4 + 2 + 1^{+2}], -3}} \subseteq W(\Dn[6] \times \Cn[1])\).  The elements \(\norm_{\Dn[4 + 2 + 1^{+2}], 3} = \ncycle(1)\cycle(3 5)\ncycle(4)\) and \(\norm_{\Dn[4 + 2 + 1^{+2}], -3} = \ncycle(3)\ncycle(5)\) have \(8\) orbits on \dEnperp[8], each of size \(2\) and none symmetric, hence acts with sign \(1\) on \dEnperp for \(n \in \sset{7, 8}\).  The transformation \(\norm_{\Dn[4 + 2 + 1^{+2}], q}\) of \dD acts on \dDC as \(\norm_{\Dn[4 + 2], q}\); so, by \eqref{eq:BCD:norm}, it acts in each case with sign
\[
(-1)^{2(2 - 1)(q - 1)/4}\Legend{2\cdot4}q\Legend{2\cdot2}q = (-1)^{(q - 1)/2}\Legend8 q = \begin{cases}
1,  & q = 3,  \\
-1, & q = -3,
\end{cases}
\]
hence with the same sign on \dEn for \(n \in \sset{7, 8}\).

\subsection{\texorpdfstring{$\mathsf D_4(a_1)^2$}{D4(a1)2}}
\label{sec:D4(a1)xD4(a1)-in-E8}

By Remark \ref{rem:Dn-as}, the element \(w \ldef w_{\Dn[2^{+4}]} = \ncycle(0 1)\ncycle(2 3)\ncycle(4 5)\ncycle(6 7)\) is attached to the admissible diagram
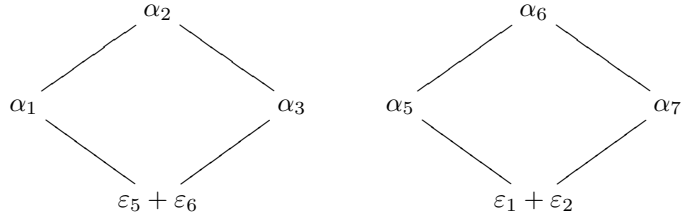
\[\xymatrix{
& \root_2 \ar@{-}[dr] & & & \root_6 \ar@{-}[dr] \\
\root_1 \ar@{-}[ur]\ar@{-}[dr] & & \root_3 & \root_5 \ar@{-}[ur]\ar@{-}[dr] & & \root_7 \\
& \basis_5 + \basis_6 \ar@{-}[ur] & 
& & \basis_1 + \basis_2 \ar@{-}[ur] 
}\]
of type \(\Dn[4](a_1)^2\).  By Remark \ref{rem:char-values}, the element \(\dot w_{\Dn[2^{+4}]}\) acts on \(\Lie(G_{\En[8]})/\Lie(G_{\Dn[8]})\) with eigenvalues the \(4\)th roots of unity, each with multiplicity \(32\).  In particular, by dimension counting, \(\card{\sgen w\bslash((\En[8])_w \setminus (\Dn[8])_w)} = 32\).  The element \(w\) has \(32\) orbits on \Enperp[8], each symmetric of size \(4\).  Thus, since \((\Dn[8])_w = \Dn[8]\) by Remark \ref{rem:Rw-full}, we have \((\En[8])_w = \En[8]\).  Put \(\dEnperp[8] = \sgen w\bslash\Enperp[8]\).

The group \(C_{W(\Dn[8])}(w)\) is the kernel of \(\sgn\textsub{\Bn[8], short}\) on
\[
(\sgen{\ncycle(0 1)} \times \sgen{\ncycle(2 3)} \times \sgen{\ncycle(4 5)} \times \sgen{\ncycle(6 7)}) \rtimes \sgen{\cycle(0 2)\cycle(1 3), \cycle(2 4)\cycle(3 5), \cycle(4 6)\cycle(5 7)}.
\]
Since \(\sgen{\cycle(0 2)\cycle(1 3), \cycle(2 4)\cycle(3 5), \cycle(4 6)\cycle(5 7)} \cong \Sgp_3\), it has order \(2^{10}\dotm3\).  By \carter*{Table 11}, contains a Sylow \(2\)-subgroup of \(C_{W(\En[8])}(w)\).
Each square of \ncycle(0 1), \ncycle(2 3), \ncycle(4 5), and \ncycle(6 7) acts on \dEnperp[8] with \(16\) orbits, each of size \(2\); and each product of two different cycles acts on \dEnperp[8] with \(8\) orbits, each of size \(4\).  The element \(\cycle(0 3)\cycle(1 2)\) (and its \(C_{W(\Dn[8])}(w)\)-conjugates \(\cycle(2 5)\cycle(3 4)\) and \(\cycle(4 6)\cycle(5 7)\)) acts on \dEnperp[8] with \(12\) orbits of size \(2\) and \(8\) orbits of size \(1\) (namely, those through \(\frac1 2(\pm\root_1 + \root_3 + 2\root_4 + 3\root_5 + 4\root_6 + 2\root_7 + 3\root_8)\), \(\frac1 2(\pm\root_1 + \root_3 + 2\root_4 + \root_5 + 2\root_6 + \root_8)\), \(\frac1 2(\root_1 \pm (\root_3 + 2\root_4 + \root_5) - \root_8)\), and \(\frac1 2(\pm\root_1 + \root_3 + \root_4 + 3\root_5 + 2\root_6 + 2\root_7 + \root_8)\)).  By \eqref{eq:BCD:Coxeter} and \eqref{eq:BCD:switch}, \(\ncycle(0 1)^2\), \(\ncycle(0 1)\ncycle(3 2)\), and \(\cycle(0 3)\cycle(1 2)\) act on \dD with sign \(1\).  Thus, each of a set of normal generators of \(C_{W(\Dn[8])}(w)\), which contains a Sylow \(2\)-subgroup of \(C_{W(\En[8])}(w)\), acts with sign \(1\) on \dEn[8]; so every element of \(C_{W(\En[8])}(w)\) acts with sign \(1\) on \dEn[8].

The group \(N_{W(\En[8])}(\sgen w)\) is the semi-direct product of \(C_{W(\En[8])}(w)\) and the group generated by \(\norm_{\Dn[2^{+4}], -1} = \ncycle(1)\ncycle(2)\ncycle(5)\ncycle(7) \in W(\Dn[8])\).  The element \(\norm_{\Dn[2^{+4}], -1}\) acts on \dEnperp[8] with \(12\) orbits of size \(2\) and \(8\) orbits of size \(1\) (namely, those through \(\frac1 2(\pm\root_1 + \root_3 + 2\root_4 + 3\root_5 + 4\root_6 + 2\root_7 + 3\root_8)\), \(\frac1 2(\pm\root_1 + \root_3 + 2\root_4 + 3\root_5 + 4\root_6 + 2\root_7 + \root_8)\), and \(\frac1 2(\pm\root_1 + \root_3 + 2\root_4 + \root_5 \pm \root_8)\)).  By \eqref{eq:BCD:norm}, it acts on \dD with sign
\[
(-1)^{4(4 - 1)(-1 - 1)/4}\Legend{2\dotm2}{-1}^4 = 1.
\]

\subsection{\texorpdfstring{$\mathsf D_8(a_3)$}{D8(a3)}}
\label{sec:D8(a3)-in-E8}

By Remark \ref{rem:Dn-as}, the element \(w \ldef w_{\Dn[4^{+2}]} = \ncycle(0 1 2 3)\ncycle(4 5 6 7)\) is attached to the admissible diagram
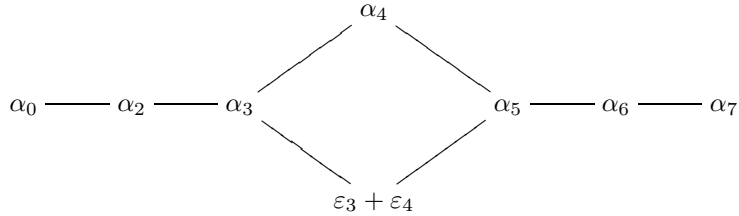
\[\xymatrix{
& & & \root_4 \ar@{-}[dr] \\
\root_0 \ar@{-}[r] & \root_2 \ar@{-}[r] & \root_3 \ar@{-}[ur]\ar@{-}[dr] & & \root_5 \ar@{-}[r] & \root_6 \ar@{-}[r] & \root_7 \\
& & & \basis_3 + \basis_4 \ar@{-}[ur] 
}\]
of type \(D_8(a_3)\).  By Remark \ref{rem:char-values}, the element \(\dot w_{\Dn[4^{+2}]}\) acts on \(\Lie(G_{\En[8]})/\spin(V_{\Dn[8]})\) with eigenvalues precisely the \(8\)th roots of unity, each with multiplicity \(16\).  In particular, by dimension counting, \(\card{\sgen w\bslash((\En[8])_w \setminus (\Dn[8])_w)} = 16\).  The element \(w\) has \(16\) orbits on \Enperp[8], each symmetric of size \(8\).  Thus, since \((\Dn[8])_w = \Dn[8]\) by Remark \ref{rem:Rw-full}, we have \((\En[8])_w = \En[8]\).  Put \(\dEnperp[8] = \sgen w\bslash\Enperp[8]\).

The group \(C_{W(\Dn[8])}(w)\) is the kernel of \(\sgn\textsub{\Bn[8], short}\) on
\[
(\sgen{\ncycle(0 1 2 3)} \times \sgen{\ncycle(4 5 6 7)}) \rtimes \sgen{\cycle(0 4)\cycle(1 5)\cycle(2 6)\cycle(3 7)}.
\]
It has order \(2^6\).  By \carter*{Table 11}, it is a Sylow \(2\)-subgroup of \(C_{W(\En[8])}(w)\).
The element \(\ncycle(0 1 2 3)^2\) (and its conjugate \(\ncycle(4 5 6 7)^2\)) acts on \dEnperp[8] with \(4\) orbits, each of size \(4\).  The element \(\cycle(0 4)\cycle(1 5)\cycle(2 6)\cycle(3 7)\) acts on \dEnperp with \(6\) orbits of size \(2\), and \(4\) orbits of size \(1\)
(namely, those through \(\frac1 2(\root_1 \pm (\root_3 + 2\root_4 + \root_5) + \root_8)\) and \(\frac1 2(\root_1 + 2\root_2 + \root_3 \pm (\root_5 + \root_8))\)).  By \eqref{eq:BCD:Coxeter} and \eqref{eq:BCD:switch}, they act on \dD with sign \(1\).  Of course, \(w = \ncycle(0 1 2 3)\ncycle(4 5 6 7)\) acts on \dEn[8] with sign \(1\).  Thus, each of a set of normal generators of \(C_{W(\Dn[8])}(w)\), which is a Sylow \(2\)-subgroup of \(C_{W(\En[8])}(w)\), acts with sign \(1\) on \dEn[8]; so every element of \(C_{W(\En[8])}(w)\) acts with sign \(1\) on \dEn[8].

The group \(N_{W(\En[8])}(\sgen w)\) is the semi-direct product of \(C_{W(\En[8])}(w)\) and \(\smashed\sgen{\norm_{\Dn[4^{+2}], 3}} \times \smashed\sgen{\norm_{\Dn[4^{+2}], -3}}\).  The elements \(\norm_{\Dn[4^{+2}], 3} = \cycle(1 3)\ncycle(2)\cycle(5 7)\ncycle(6)\) and \(\norm_{\Dn[4^{+2}], -3} = \ncycle(1)\ncycle(3)\ncycle(5)\ncycle(7)\) act on \dEnperp[8] with \(8\) orbits, each of size \(2\).  By \eqref{eq:BCD:norm}, the element \(\norm_{\Dn[4^{+2}], q}\) acts on \dD with sign
\[
(-1)^{2(2 - 1)(q - 1)/4}\Legend{2\dotm4}q^2 = (-1)^{(q - 1)/2}.
\]


\end{document}